\pdfoutput=1
\documentclass[12pt]{article}
\usepackage{amsthm,amsmath,amsfonts,mathrsfs,amssymb,systeme,bm, bbm}
\usepackage{color,xcolor,graphicx}
\usepackage{natbib}
\usepackage{enumerate}
\usepackage[ruled,vlined]{algorithm2e}

\usepackage{mlmodern}
\usepackage{pdfrender,xcolor}
\usepackage{graphicx,color}
\usepackage[hidelinks,colorlinks=true,linkcolor=blue,citecolor=blue,urlcolor=blue]{hyperref}
\usepackage{algpseudocode}

\usepackage{comment}
\usepackage{cancel}

\newcommand{\R}{\mathbb{R}}
\newcommand{\D}{\mathcal{D}} 
\newcommand{\Z}{\mathbb{Z}}
\newcommand{\N}{\mathbb{N}}
\newcommand{\E}{\mathbb{E}}
\newcommand{\J}{\mathcal{J}} 
\newcommand{\A}{\mathcal{A}}
\newcommand{\Prob}{\mathbb{P}}

\newcommand{\I}[1]{\mathbbm{1}{\left( #1 \right)}} 
\newcommand{\T}{\prime} 
\newcommand{\up}[1]{{(#1)}}
\newcommand{\e}[1]{e^{\up{#1}}}
\newcommand{\cc}{r} 
\newcommand{\uu}{{\up{\cc}}} 
\newcommand{\inner}[1]{\langle #1 \rangle}

\providecommand{\abs}[1]{\left\lvert#1\right\rvert}

\providecommand{\norm}[1]{\lVert#1\rVert}
\providecommand{\Norm}[1]{\left\lVert#1\right\rVert}
\newtheorem{theorem}{Theorem}
\newtheorem{lemma}[theorem]{Lemma}

\theoremstyle{definition}
\numberwithin{equation}{section}

\newtheorem{remark}{Remark}

\newtheorem{proposition}[theorem]{Proposition}
\newtheorem{assumption}{Assumption}

\title{Steady-State Convergence of the Continuous-Time Routing System with General Distributions in Heavy Traffic}

\author{Jin Guang$^1$\\{\small \href{mailto:jinguang@link.cuhk.edu.cn}{jinguang@link.cuhk.edu.cn}} \and Yaosheng Xu$^2$\\ {\small \href{mailto:yaosheng.xu@chicagobooth.edu}{yaosheng.xu@chicagobooth.edu}} \and J. G. Dai$^3$\\ {\small \href{mailto:jd694@cornell.edu}{jd694@cornell.edu}}   }

\date{{
$^1$School of Data Science, The Chinese University of Hong Kong, Shenzhen, China}\\%
$^2$Booth School of Business, University of Chicago, Chicago, Illinois\\
$^3$ School of Operations Research and Information Engineering, Cornell University, Ithaca, New York} 

\begin{document}




\maketitle

\begin{abstract}
This paper examines a continuous-time routing system with general interarrival and service time distributions, operating under the join-the-shortest-queue and power-of-two-choices policies. Under a weaker set of assumptions than those commonly found in the literature, we prove that the scaled steady-state queue length at each station converges weakly to an identical exponential random variable in heavy traffic. Specifically, our results hold under the assumption of the $(2 + \delta_0)$th moment for the interarrival and service distributions with some $\delta_0 > 0$. The proof leverages the Palm version of the basic adjoint relationship (BAR) as a key technique.

\end{abstract}


\section{Introduction}

We study a continuous-time queueing system consisting of a single arrival source and $J$ parallel service stations, each equipped with a single server and an infinite waiting queue. Jobs arrive according to a renewal process, and service times at each station are independent and identically distributed (i.i.d.)~with general distributions. When a job arrives at the system, it is immediately routed to one of the $J$ stations based on the information of all queue lengths. Upon arrival, the system implements one of two routing policies: the join-the-shortest-queue (JSQ) policy, as explored by \citet{Wins1977} and \citet{Webe1978}, or the power-of-two-choices (Po2) policy, as discussed by \citet{VvedDobrKarp1996} and \citet{Mitz1996}. Under the JSQ policy, jobs are routed to the server with the shortest queue, promoting queue length equalization and minimizing waiting times. Alternatively, the Po2 policy routes jobs to the shorter of two randomly selected servers, offering a simpler decision-making procedure while still maintaining effective load balancing. In practice, the choice between JSQ and Po2 policies enables a trade-off between optimal queue length balancing and reduced routing overhead, depending on system requirements.

In this paper, we demonstrate that the scaled steady-state queue length at each station converges weakly to a common exponential random variable in heavy traffic. Specifically, we consider a sequence of routing systems indexed by $r \in (0, 1)$. Under the heavy traffic condition parameterized by $r \to 0$ with a fixed number of servers $J$, we prove that if the interarrival and service times have finite $(2+\delta_0)$th moments for some $\delta_0>0$, then
\begin{equation*} 
    \left( rZ^\uu_1, \ldots, rZ^\uu_J \right) \Rightarrow \left( Z^*, \ldots, Z^* \right),
\end{equation*}
where $\Rightarrow$ denotes the weak convergence, $Z^\uu_j$ denotes the steady-state queue length at station $j$ in the $r$th system, and $Z^*$ is an exponential random variable. This result hinges on the observation that the steady-state queue length vector collapses onto the line where all queue lengths are equal, with deviations from this line uniformly bounded. This phenomenon in the queueing systems is known as state-space collapse (SSC).

Prior work on state-space collapse and heavy traffic limits for routing systems has been conducted under the discrete-time framework, including the drift method \citep{EryiSrik2012, MaguSrikYing2014}, transform method \citep{HurtMagu2020} and the Stein's method \citep{ZhouShro2020}. However, these studies heavily rely on the assumption that the interarrival and service times have bounded support, thereby implying the boundedness of all moments.  In our work, we relax this assumption and consider a continuous-time system, demonstrating that the existence of only the $(2+\delta_0)$th moments of the interarrival and service is sufficient to guarantee steady-state convergence in heavy traffic. 

Our result is built on a novel methodology known as the basic adjoint relationship (BAR) approach. One of the key advantages of the BAR approach is its ability to directly characterize the stationary distribution of a queueing system, thereby bypassing the discussion on its transient dynamics. This approach has been successfully applied in recent studies \citep{BravDaiMiya2017, BravDaiMiya2023, DaiGlynXu2023, GuanChenDai2024} to establish SSC or weak convergence in various queueing systems.

The paper is organized as follows. In Section \ref{sec: setting}, we introduce the routing system configured with either the JSQ or Po2 policy and present the heavy traffic framework. Section \ref{sec: result} presents our main results, where we establish state-space collapse (SSC) and weak convergence of the scaled steady-state queue lengths. The BAR approach, which underpins our methodology, is introduced in Section \ref{sec: bar}. The proof of SSC is provided in Section \ref{sec: ssc}, followed by the proof of weak convergence in Section \ref{sec:weak}.

\section{Model} \label{sec: setting}
We use $\N$ to denote the set of positive integers $\{1,2,\ldots\}$. For $a,b\in \R$, let $a\wedge b\equiv \min\{a,b\}$ and $a\vee b\equiv \max\{a,b\}$. For a vector $a\in \R^J$, let $a_{\min}\equiv \min_{j=1,\ldots,J} a_j$, $a_{\max}\equiv \max_{j=1,\ldots,J} a_j$, and $a_{\Sigma}\equiv \sum_{j=1}^J a_j$. Additionally, $\e{j}$ represents a $J$-dimensional vector with the $j$th element being $1$ and all other elements being $0$. 

We consider a system with one source of arrival routing to $J$ parallel stations, indexed by $j\in \J \equiv \{1, \ldots, J\}$. We denote an i.i.d.~sequence of random variables $\{T_{e}(i), i\in \N\}$ and a real number $\alpha>0$ to be associated with the arrival, and an i.i.d.~sequence of random variables $\{T_{s,j}(i), i\in \N\}$ and a real number $\mu_j>0$ associated with each station $j \in \J$. All of the above are defined on a common probability space $(\Omega, \mathscr{F}, \mathbb{P})$. We assume these $J+1$ sequences
\begin{equation} \label{eq: unitized}
    \{T_{e}(i), i\in \N\}, \quad \{T_{s,j}(i), i\in \N\}_{j\in \J},
\end{equation}
are independent and unitized, meaning $\mathbb{E}[T_{e}(1)]=1$ and $\mathbb{E}[T_{s,j}(1)]=1$ for all $j\in \J$. For each $i \in \N$, we denote by $T_e(i)/\alpha$ the interarrival time between the $i$th and $(i + 1)$th arriving jobs, and by $T_{s,j}(i)/\mu_j$ the service time of the $i$th job at station $j$. Here, $\alpha$ represents the arrival rate, and $\mu_j$ is the service rate at station $j$. We assume the following moment conditions on the interarrival and service time distributions.

\begin{assumption} \label{ass: moment condition}
    The interarrival and service times have finite $(2+\delta_0)th$ moments for some $\delta_0 > 0$. Specifically, we assume
    \begin{equation*} 
        \E\left[T_{e}^{ 2+\delta_0  }{(1)}\right]<\infty, \quad \E\left[T_{s,j}^{ 2+\delta_0}{(1)}\right]<\infty\text{ for  }j \in \J.
    \end{equation*}
\end{assumption}

We denote by $Z_j(t)$ the queue length at station $j$ at time $t\geq 0$, including the possibly one in service, and by $Z(t)\equiv(Z_1(t),\ldots,Z_J(t))$ the vector of queue lengths. The routing decision for an arriving job at time~$t$ is given by
\begin{equation} \label{eq: ut}
    u(t) = a(Z(t_-), U(N_e(t))),
\end{equation}
where $a:\Z_+^J \times [0,1] \to \J$ is a function at time~$t$, $Z(t_-)$ represents the queue length at time~$t$ before the routing decision is made, $N_e(t)$ denotes the number of arrivals during $(0,t]$, which is also the number of decisions made by the system manager, and $\{U(i), i\in \N\}$ is a sequence of i.i.d.~random variables uniformly sampled on $[0,1]$. By allowing the decision $u(t)$ to depend on $U$, we are able to align the system with two randomized routing policies. In this paper, we focus on the join-the-shortest-queue (JSQ) policy and the power-of-two-choices (Po2) policy.  

The JSQ policy assigns the arriving job to the station with the shortest queue length, i.e., $u(t) = \operatorname*{argmin}_{j\in \J} Z_j(t_-)$, where ties are broken arbitrarily. The JSQ policy requires knowledge of the queue lengths at all stations, which can be costly when the number of stations is large in some communication networks. An alternative algorithm is the Po2 policy, which assigns the arriving job to the station with the shorter queue length between two randomly chosen stations. Specifically, $u(t) = \operatorname*{argmin}_{j\in \{j_1, j_2\}} Z_j(t_-)$, where $j_1, j_2$ are two random stations sampled from $\J$ with $j_1\neq j_2$. The sampling follows probabilities proportional to their service rates. Ties are broken arbitrarily. The probability of choosing station $j$ is $p_j\equiv \mu_j/\mu_{\Sigma}$ and the probability of choosing station $j_1$ and $j_2$ with $j_1\neq j_2$ is $p_{j_1 j_2} \equiv p_{j_1} {p_{j_2}}/{(1-p_{j_1})} + p_{j_2} {p_{j_1}}/({1-p_{j_2}})$. The Po2 policy only requires queue length information from two stations, making it more practical for large-scale systems.

The routing system can be modeled as a Markov process as follows. For time $t \geq 0$, we denote by $R_e(t)$ the residual time until the next arrival to the system and denote 
\begin{equation} \label{eq: R_s}
	R_{s,j}(t) = \begin{cases}
		\text{residual service time for the job being processed at station $j$}, & \text{if } Z_j(t) > 0, \\
		\text{service time of the next job to be processed at station $j$}, & \text{if } Z_j(t) = 0.
	\end{cases}
\end{equation}
We write $R_s(t)$ as a $J$-dimensional vector whose $j$th element is $R_{s,j}(t)$. For any given $t \geq 0$, we set
\begin{equation*}
    X(t) = \left( Z(t), R_e(t), R_s(t) \right).
\end{equation*}
Note that the remaining times in $X(t)$ are solely used for analysis purposes, and the decision makers depend only on the queue length information in \eqref{eq: ut}. In fact, they may not have access to other detailed information beyond the queue length. $\{X(t), t \geq 0\}$ is a Markov process with respect to the filtration $\mathbb{F}^X = \{\mathcal{F}^X_t, t \geq 0\}$ defined on the state space $\mathbb{S} = \Z_+^J \times \R_+ \times \R_+^J$, where $\mathcal{F}^X_t = \sigma(\{X(s), 0 \leq s \leq t\})$. We assume that each sample path of the process $\{X(t), t \geq 0\}$ is right-continuous with left limits.

To carry out the heavy traffic analysis, we consider a sequence of routing systems indexed by $r\in (0,1)$. To keep the presentation clean, we set the arrival rate as the only parameter depending on $r$ and denote by $\alpha^\uu$ the arrival rate for the $r$th system. Specifically, we denote
\begin{equation} \label{eq: heavy traffic}
    \alpha^\uu = \mu_{\Sigma} - r ,
\end{equation}
All the other parameters are assumed to be independent of $r$, including the service rates $\{\mu_j\}_{j\in\J}$, unitized interarrival and service times specified in~\eqref{eq: unitized}. Then the traffic intensity $$\rho^\uu \equiv \alpha^\uu / \mu_{\Sigma} \to~1, \quad \text{as } r\to 0,$$ which means the system is in heavy traffic. We then denote by $\{X^\uu(t), t \geq 0\}$ the corresponding Markov process in the $r$th system. Our discussion is based on the steady-state behavior, which motivates us to make the following assumption. The following assumption can be proved under some mild distributional assumptions on interarrival times; see \citet{Dai1995} and \citet{Bram2011}.

\begin{assumption} \label{ass: stability}
    For each $r \in (0, 1)$, the Markov process $\{X^\uu(t), t \geq 0\}$ is positive Harris recurrent and it has a unique stationary distribution $\pi^\uu$.
\end{assumption}

For $r \in (0, 1)$, we denote by 
$$X^\uu = (Z^\uu, R^\uu_e, R^\uu_{s} )$$ 
the random vector that follows the stationary distribution. To simplify the notation, we use $\E_{\pi}[\cdot]$ (rather than $\E_{\pi^\uu}[\cdot]$) to denote the expectation with respect to the stationary distribution when the index $r$ is clear from the context.

\section{Main Results} \label{sec: result}

In this section, we demonstrate our main result: for either JSQ or Po2 policy, the vector of the scaled steady-state queue length $rZ^\uu$ weakly converges to a vector whose elements are the same exponential random variable $Z^*$ in heavy traffic.

\begin{theorem} \label{thm: main result}
    Suppose Assumptions \ref{ass: moment condition} and \ref{ass: stability} hold. Under the JSQ or Po2 policy, as $r\to 0$, we have
	\begin{equation*} 
		\left( rZ^\uu_1, \ldots, rZ^\uu_J \right) \Rightarrow \left( Z^*, \ldots, Z^* \right),
	\end{equation*}
    where $Z^*$ is an exponential random variable with mean
    \begin{equation} \label{eq: mean}
        m = \frac{1}{2J}\sum_{j\in \J}\mu_j\left(c_e^2  + c_{s,j}^2 \right).
    \end{equation}
    Here, $c_e^2$ is the squared coefficient of variation (SCV) of the interarrival time, and $c_{s,j}^2$ is the SCV of the service time at station $j$. 
\end{theorem}

We recall that for a positive random variable $T$, its SCV, denoted as $c^2(T)$, is defined as $c^2(T) = {\operatorname*{var}(T)}/{(\E[T])^2}$.

To prove Theorem \ref{thm: main result}, we first establish SSC and weak convergence of the scaled average queue length as follows. To establish SSC, we show that the queue lengths concentrate along the vector $e\equiv (1,\ldots,1)$. We denote the components of the vector $z$ that are parallel and perpendicular to the vector $e$ by
\begin{equation*} 
    z_{\parallel} = \frac{\inner{z,e}}{\norm{e}^2}e=\bar{z} \cdot e , \quad z_{\perp} = z - z_{\parallel} = (z_j - \bar{z})_{j\in \J},
\end{equation*}
where the average $\bar{z} \equiv z_{\Sigma} /J= \sum_{j\in \J} z_j /J$ and the norm $\norm{\cdot}$ represents the Euclidean norm. 

\begin{proposition}[State-Space Collapse] \label{prop: ssc}
    Suppose Assumptions \ref{ass: moment condition} and \ref{ass: stability} hold.  Under the JSQ or Po2 policy, the difference between the queue length and the average queue length is uniformly bounded, i.e.,
    \begin{equation*} 
        \sup_{r\in (0,r_0)} \E\left[ \Norm{Z^\uu_{\perp}}^{1 + \delta_0 / (1 + \delta_0)} \right] < \infty,
    \end{equation*}
    where $r_0$ is a constant in $(0,1)$ depending on the policy. 
\end{proposition}

\begin{remark}
    Proposition \ref{prop: ssc} is sufficient to support Theorem~\ref{thm: main result}. In general, if the moment condition in Assumption \ref{ass: moment condition} is strengthened to $M+\delta_0$, this SSC result can be extended to $M+\delta_0/(M+\delta_0)$, as discussed in Section \ref{sec: ssc}.
\end{remark}

\begin{remark}
	The choice of $r_0$ is related to the routing policy. Specifically, we set $r_0=1\wedge\mu_{\min} / 2$ for the JSQ policy and set $r_0=1\wedge\mu_{\min}/2\wedge \mu_\Sigma/(1 + 2\sqrt{J}/\min_{j_1<j_2}p_{j_1j_2})$ for the Po2 policy.
\end{remark}

\begin{proposition} \label{prop: weak convergence}
    Suppose Assumptions \ref{ass: moment condition} and \ref{ass: stability} hold.  Under the JSQ or Po2 policy, as $r\to 0$, we have
    \begin{equation*} 
        r \bar{Z}^\uu \Rightarrow Z^*,
    \end{equation*}
    where $Z^*$ is an exponential random variable defined in \eqref{eq: mean}.
\end{proposition}

Proposition \ref{prop: ssc} and Markov's inequality imply that for any station $j\in \J$, $rZ_j^\uu - r\bar{Z}^\uu$ converges to $0$ in probability. Consequently, Theorem \ref{thm: main result} follows directly from Proposition~\ref{prop: weak convergence}.

Throughout the rest of the paper, our proof utilizes the BAR, which will be introduced in Section~\ref{sec: bar}. The proofs of Propositions \ref{prop: ssc} and \ref{prop: weak convergence} will be presented in Sections \ref{sec: ssc} and \ref{sec:weak}, respectively.

\section{Basic Adjoint Relationship} \label{sec: bar}
For our analysis, we introduce the Basic Adjoint Relationship (BAR) of the routing system, which enables us to directly characterize the stationary distribution of the Markov processes $\{X^\uu(t), t \geq 0\}$.

Let $\D$ be the set of bounded functions $f: \mathbb{S} \rightarrow \mathbb{R}$ satisfying the following conditions: for any $z \in \mathbb{Z}_{+}^J$, (a) the function $f(z,\cdot,\cdot):\R_+\times \R_+^J\rightarrow \R$ is continuously differentiable at all but finitely many points; (b) the derivatives of $f(z,\cdot,\cdot)$ in each dimension have a uniform bound over~$z$.

For a routing policy whose Markov process $\{X(t),t\geq 0\}$ is defined in Section \ref{sec: setting}, the stationary distribution $\pi$ satisfies the following balance equation:
\begin{equation*}
    \E_{\pi}\left[ f(X(1)) - f(X(0)) \right] = 0, \quad \forall~ f\in \D,
\end{equation*}
where $X(0)$ follows the stationary distribution $\pi$.
Hence, by applying the fundamental theorem of calculus to expand $f(X(1))-f(X(0))$ and introducing the concept of the Palm probability measure, a conditional distribution of a point process given an event occurs at a specified time, we obtain the BAR for the routing system as similarly proposed in (6.16) of \citet{BravDaiMiya2023}: for any $f\in \D$,
\begin{align}\label{eq:BAR24}
    \E_{\pi}\left[  \A f(X ) \right] + \alpha \E_{e} \left[   f(X_+^{(e)}) - f(X_-^{(e)})  \right] + \sum_{j\in \J} \lambda_j \E_{s,j} \left[   f(X_+^{(s,j)}) -  f(X_-^{(s,j)})  \right] = 0,
\end{align}
where $\lambda_j=\mu_j \mathbb{P}(Z_j>0)$ is the departure rate at station $j$.
Here for the first term of \eqref{eq:BAR24}, $X$ follows the stationary distribution $\pi$ and the ``interior operator'' $\A$ is defined as
\begin{equation} \label{eq: A}
    \A f(x)=-\frac{\partial f}{\partial r_{e}}( x)-\sum_{j \in \J} \frac{\partial f}{\partial r_{s,j}}( x) \I{z_j>0},\quad x=(z,r_{e},r_{s})\in \mathbb{S},
\end{equation}
In the absence of external arrival or service completion events, the first term in \eqref{eq: A} indicates that the residual times for external arrivals decrease at a rate of \(-1\), while the second term in \eqref{eq: A} indicates that when the corresponding station has a positive number of jobs, the residual service times decrease at a rate of \(-1\); otherwise, the decreasing rate is \(0\). From this perspective, the operator $\A f$ in \eqref{eq: A} corresponds to the evolution of the routing system between jumps.

The remaining terms in \eqref{eq:BAR24} describe the state changes resulting from external arrivals and service completions, utilizing the Palm probability measure. Mathematically, the Palm probability measure is a conditional distribution of a general point process given it has a point at a specific location; see, for instance, \citet{BaccBrem2003} and \citet{Miya1994}. In particular, in the second term of \eqref{eq:BAR24}, $$(X_-^{(e)}, X_+^{(e)}) \equiv ((Z_-^{(e)}, R_{-,e}^{(e)}, R_{-,s}^{(e)}), (Z_+^{(e)}, R_{+,e}^{(e)}, R_{+,s}^{(e)}))$$ represents the pre-jump and post-jump states conditioned on an external arrival at the system occurring at time $0$. Note that the choice of time $0$ is arbitrary since $X(0)$ follows the stationary distribution. We denote the Palm measure and its associated expectation by \(\mathbb{P}_{e}\) and \(\mathbb{E}_{e}\), respectively. When an external arrival occurs, \(R_{-,e}^{(e)}\) decreases to \(0\) and \(R_{+,e}^{(e)}\) is reset to a new independent interarrival time. Following the routing policy, the arriving job is routed to the station \(u\equiv a(Z_-^{(e)}, U)\), where $U$ is a random variable uniformly distributed on $[0,1]$. Hence, the queue length of station $u$ increases by one, i.e., \(Z_{+,u}^{(e)} = Z_{-,u}^{(e)} + 1\), while the states of all other stations remain unchanged.

The third term of \eqref{eq:BAR24} is defined similarly, where $$ (X_-^{(s,j)}, X_+^{(s,j)}) \equiv ((Z_-^{(s,j)}, R_{-,e}^{(s,j)}, R_{-,s}^{(s,j)}),  (Z_+^{(s,j)},R_{+,e}^{(s,j)}, R_{+,s}^{(s,j)}))$$ 
represents the pre-jump and post-jump states conditioned on a service completion at station $j$ occurring at time $0$. The Palm measure and its associated expectation are denoted by \(\mathbb{P}_{s,j}\) and \(\mathbb{E}_{s,j}\), respectively. When a service completion occurs at station \(j\), \(R_{-,s,j}^{(s,j)}\) decreases to \(0\), \(R_{+,s,j}^{(s,j)}\) is reset to an independent service time, and the queue length at station \(j\) decreases by one, i.e., \(Z_{+,j}^{(s,j)} = Z_{-,j}^{(s,j)} - 1\), and all other states remain unchanged.

The precise definitions of the Palm measures are provided in (6.8) and (6.9) of \citet{BravDaiMiya2023}, and the properties of the pre-jump and post-jump states are detailed in the following lemma.

\begin{lemma}[Lemma~6.3 of \citet{BravDaiMiya2023}]\label{lmm: independent}
    The pre-jump and the post-jump states have the following representation,
    $$
    \begin{aligned}
    & X_{+}^{(e)}=X_{-}^{(e)}+\sum_{j\in\J}\left( \e{j},    T_{e} /\alpha, 0 \right)\I{a(Z_-^{(e)}, U)=j},  \text { under } \mathbb{P}_{e}, \\
    & X_{+}^{(s,j)}=X_{-}^{(s,j)}+\left( - \e{j}  , 0,   \e{j}T_{s, j} /\mu_j \right), \quad \text { under } \mathbb{P}_{s, j},~j \in \J,
    \end{aligned}
    $$
    where $T_{e},T_{s, j}$ for $j \in \J$ are random variables defined on the measurable space $(\mathbb{S}^2, \mathscr{B}(\mathbb{S}^2))$, such that, under Palm distribution $\mathbb{P}_{e}$, $T_{e}$ is independent of $X_{-}^{(e)}$ and follows the same distribution as $T_{e}(1)$ on $(\Omega, \mathscr{F}, \mathbb{P})$, and, under Palm distribution $\mathbb{P}_{s, j}$, $T_{s, j}$ is independent of $X_{-}^{(s,j)}$ and has the same distribution as that of $T_{s, j}(1)$ on $(\Omega, \mathscr{F}, \mathbb{P})$.
\end{lemma}

To simplify the notation, we omit the superscripts “\(^{(e)}\)” or “\(^{(s,j)}\)” for \(j \in \mathcal{J}\). Then, the BAR introduced by \eqref{eq:BAR24} becomes: for any $f\in \D$,
\begin{equation*}
	\begin{aligned}
		\E_{\pi}\left[  \A f(X ) \right] &+   \alpha \E_{e} \Bigg[  \sum_{j\in \J} f(X_- + \Delta_{e,j})\I{a(Z_-,U)=j} -f(X_-) \Bigg] \\
		&+ \sum_{j\in \J} \lambda_j \E_{s,j} \left[ f(X_-+\Delta_{s,j})-f(X_-)  \right] = 0.
	\end{aligned}
\end{equation*}
Conditioning on the state $X_-=x_-\equiv (z_-,r_{-,e},r_{-,s})$, the randomness of the decision $u=a(z_-,U)$ depends on $U$. Thus, we can simplify the BAR by conditioning on $U$ as follows:
\begin{equation}\label{eq:BAR}
	\begin{aligned}
		-\E_{\pi}\left[  \A f(X ) \right] &=   \alpha \E_{e} \Bigg[  \sum_{j\in \J} f(X_- + \Delta_{e,j})p_{j|Z_-} -f(X_-)\Bigg] + \sum_{j\in \J} \lambda_j \E_{s,j} \left[ f(X_-+\Delta_{s,j})-f(X_-)  \right],
	\end{aligned}
\end{equation}
where the conditional probability is denoted by $p_{j|z_-} \equiv \Prob(a(Z_-,U)=j\mid Z_-=z_-)$.

\section{Proof of State Space Collapse in Proposition \ref{prop: ssc}} \label{sec: ssc}

In this section, we present the proof of state space collapse in Proposition \ref{prop: ssc} using the BAR approach. The proof is carried out by mathematical induction, following the method of \citet{GuanChenDai2024}. We apply the BAR in~\eqref{eq:BAR} with test functions inspired by \citet{EryiSrik2012} for the JSQ policy and \citet{MaguSrikYing2014} for the Po2 policy.

Without deviating much, we will present Proposition \ref{prop: ssc high} below to present a general SSC result with the higher moment conditions. It can be seen that Proposition \ref{prop: ssc} is a special case of Proposition~\ref{prop: ssc high}, which simply requires up to $(2+\delta)$th moments to carry out our weak convergence. 

\begin{proposition} \label{prop: ssc high}
    Given an integer $M\geq 1$ and a constant $\varepsilon\in [0,1)$, if the following moment condition holds for the unitized times:
    \begin{equation}\label{eq: moment condition high}
        \E\left[T_{e}^{ M+1+\varepsilon }{(1)}\right]<\infty, \text{ and 
 }\E\left[T_{s,j}^{ M+1+\varepsilon}{(1)}\right]<\infty\text{ for  }j \in \J.
    \end{equation}
    Then, there exists a positive constant $C<\infty$, independent of $r$, such that for all $r\in(0,r_0)$,
	\begin{equation*} 
		\E_\pi\left[ \Norm{Z_{\perp}^\uu}^{M + \varepsilon / (M+\varepsilon)} \right] \leq C.
	\end{equation*}
\end{proposition}

The proof is based on mathematical induction and the BAR in \eqref{eq:BAR} with appropriately designed test functions. A proof sketch is presented in Section \ref{sec: sketch ssc}, with the detailed proof provided in Sections~\ref{sec: base}-\ref{sec:proofS4}. We first provide the proof for the integer moment condition when $\varepsilon=0$, and we extend the results to the general case when $\varepsilon>0$ in Section \ref{sec: non int}.

\subsection{Sketch of Proof} \label{sec: sketch ssc}

We first prove Proposition \ref{prop: ssc high} with $\varepsilon=0$. Our proof includes moment bounds for $\norm{Z_{\perp}^\uu}$, along with auxiliary results that bound the expectations of certain cross terms involving $\norm{Z_{\perp}^\uu}$ and residual times.

\paragraph{Induction hypotheses.} For each integer $n=0,1,\ldots,M$, there exist positive and finite constants $C_{n}, D_{n}, E_{n}, F_{n}$, independent of $r$, such that the following statements hold for all $r\in (0,r_0)$:  
\begin{itemize}
	\item[(S1)] $\E_{\pi}[\norm{Z_\perp^\uu}^{n}] \leq C_n$,
	\item[(S2)] $\E_{e}[\norm{Z_{-,\perp}^\uu}^{n}]+\sum_{\ell\in \J}\E_{{s,\ell}}[\norm{Z_{-,\perp}^\uu}^{n}] \leq D_n$,
	\item[(S3)] $\E_{\pi}[\norm{Z_\perp^\uu}^{n}\psi_{M-n}(R_e^\uu, R_s^\uu)] \leq E_n$,
	\item[(S4)] $\E_{e}[\norm{Z_{-,\perp}^\uu}^{n}\psi_{M-n}(R_{-,e}^\uu, R_{-,s}^\uu)]+ \sum_{\ell\in \J}\E_{s,\ell}[\norm{Z_{-,\perp}^\uu}^{n}\psi_{M-n}(R_{-,e}^\uu, R_{-,s}^\uu)] \leq F_n$,
\end{itemize}
where $\psi_{n}(r_e, r_s) = r_e^n + \sum_{j\in \J} r_{s,j}^n$. The function $\psi_{M-n}$, appearing in the auxiliary statements (S3)-({S4}), depends on the order $M+1$ of the moment condition. This design of the auxiliary statements is crucial to our proof, as it helps reduce the order of the moment condition required to establish uniform bounds on $\E_{\pi}[\norm{Z_\perp^\uu}^M]$.

\paragraph{Mathematical induction.} We initially prove ({S1})-({S4}) for the base step for $n=0$. The proof details are provided in Section \ref{sec: base}. For the induction step, we verify ({S1})-({S4}) for each given $n$, under the induction hypotheses that they hold for $0,\ldots,n-1$. 

\paragraph{Proof of (S1)} Different from the test function  \(g_n(x) = \norm{z_{\perp}}^{n+1}/(n+1)\) applied in the drift method by \citet{EryiSrik2012} and \citet{MaguSrikYing2014}, we need the test function to account for the remaining time in our setting. In addition to the integer-valued queue length, we add variables to depict the remaining time for the job being served and the time until a new job arrives. With a slight abuse of notation, we refer to this concept as ``workload", though we define it in units of a job rather than the traditional measure of those by the service time.

We define the workload at station $j$ as
\[
w_j(x) = z_j + \mu_j r_{s,j} - \alpha^\uu r_e p_{j|z} = \begin{cases}
(z_j - 1) + \mu_j r_{s,j} + (1- \alpha^\uu r_e p_{j|z}), & \text{ if $z_j>0$},\\
(T_{s,j} - 1) + (1- \alpha^\uu r_e p_{j|z}), & \text{ if $z_j=0$},
\end{cases}
\]
where the last equality holds due to \eqref{eq: R_s}, since \(r_{s,j}\) represents a new service time when \(z_j = 0\). This is a similar definition of the workload used in (7.2) of \citet{HongScul2023} and Definition 5.1 of \citet{GrosHongHarcSche2024}.

For \(z_j > 0\), the first term represents the number of jobs waiting in the buffer of station~\(j\). The second term reflects the expected remaining proportion of a job being processed, and the third term accounts for the expected fraction of the next arriving job, but this term is adjusted by the routing probability \(p_{j|z}\). When \(z_j = 0\), the reasoning is similar, with the expected value of the first term zero due to \(\E[T_{s,j}] = 1\).

We replace the queue length $z$ in the test function $g_n$ with the workload $w$:
\begin{align*}
	\tilde g_n(x) &= \frac{1}{n+1} \norm{w_\perp(x)}^{n+1} \equiv \frac{1}{n+1} \norm{w(x) - \bar{w}(x)\cdot e}^{n+1},
\end{align*}
where $w(x) = (w_1(x), \ldots, w_J(x))$ and $\bar w(x) = \sum_{j\in \J} w_j(x)/J$.

Since the \((n+1)\)th power of the residual times introduces unnecessary terms and worsens the moment condition for \eqref{eq: moment condition high}, we expand \(\tilde{g}_n\) and retain only the first two leading order terms of \(z_\perp = z - \bar{z}\cdot e\):
\[
\tilde f_n (x)=\frac{1}{n+1}\norm{z_{\perp}}^{n+1} -\alpha^\uu r_e\norm{z_{\perp}}^{n-1}\sum_{j\in \J} z_{\perp,j}p_{j|z}  +  \norm{z_{\perp}}^{n-1}\sum_{j\in \J} z_{\perp,j}\mu_j  r_{s,j}.
\]
After calculating and simplifying $p_{j|z}$ under different routing policies, we use the following test function for the proof of (S1):
\begin{equation} \label{eq: func_s1}
	f_{n}(x)=\frac{1}{n+1}\norm{z_{\perp}}^{n+1} -\eta_{n,e}(z_{\perp}) \cdot  \alpha^\uu r_e   + \sum_{j\in \J} \eta_{n,s,j}(z_{\perp})  \cdot \mu_j  r_{s,j},
\end{equation}
where $\eta_{n,e}$ and $\eta_{n,s,j}$ for $j\in \J$ are given by
\begin{align*}
	\eta_{n,e}(z_{\perp}) &= \begin{cases}
		\norm{z_{\perp}}^{n-1} \cdot \min_{j\in\J}z_{\perp,j} & \text{ for JSQ policy},\\
		\norm{z_{\perp}}^{n-1} \cdot \left( \sum_{j\in \J}  {p_{j}}z_{\perp,j}  - \frac{p_{\min}^2}{\sqrt{J}}\Norm{z_{\perp}} \right) & \text{ for Po2 policy},
	\end{cases}\\
	\eta_{n,s,j}(z_{\perp}) &= \norm{z_{\perp}}^{n-1} \cdot z_{\perp,j} \quad \text{ for all } j\in \J.
\end{align*}

Applying the test function $f_n$ to the BAR \eqref{eq:BAR}, $\A f_n$ produces terms of order $\norm{z_\perp}^n$, while the Palm terms in BAR \eqref{eq:BAR} yield a polynomial in \(\norm{z_\perp}\) of order up to \(n-1\). Specifically, we have the following inequality:
\begin{align*}
	&b\E_{\pi}\left[ \Norm{Z_\perp^\uu}^n \right] \leq c_1 \E_{e}\left[ \Norm{Z^\uu_{-,\perp}}^{n-1} \right] +  c_2 \E_{e}\left[ \Norm{Z^\uu_{-,\perp}}^{n-1}\left( R_{-,e}^\uu+\sum_{j\in \J}R_{-,s.j}^\uu \right) \right]\\
	&\qquad +c_3\sum_{\ell\in \J}\E_{s,\ell}\left[ \Norm{Z^\uu_{-,\perp}}^{n-1} \right] +c_4\sum_{\ell\in \J}\E_{s,\ell}\left[ \Norm{Z^\uu_{-,\perp}}^{n-1} \left( R_{-,e}^\uu+\sum_{j\in \J}R_{-,s.j}^\uu \right)\right],
\end{align*}
where \(b, c_1, c_2, c_3, c_4\) are positive constants independent of \(r\). Hence, by the induction hypotheses (S2) and (S4) for $n-1$, the right-hand side is bounded by a constant \(B\) independent of \(r\). Therefore, the induction step $n$ for (S1) is proved. The detailed proof of (S1) is provided in Section \ref{sec: proof of S1}.

\paragraph{Proof of (S2) to (S4)} The proofs for Statements (S2)-(S4) follow a similar argument as in the proof of (S1), in which statements are bounded by applying specific test functions to the BAR \eqref{eq:BAR} and utilizing the induction hypotheses. Below we specify the test functions for (S2)-(S4) in \eqref{eq: func_s2}-\eqref{eq: func_s4}, respectively, whose complete proofs are given in Sections \ref{sec: proofS2}-\ref{sec:proofS4}.
\begin{align}
    f_{n,D}\left(x\right) &= \norm{z_{\perp}}^{n}\psi_1(r_e, r_s),\label{eq: func_s2}\\
    f_{n,E}\left(x\right) &= \norm{z_{\perp}}^{n}\psi_{M-n+1}(r_e, r_s),\label{eq: func_s3}\\
    f_{n,F}\left(x\right) &= \norm{z_{\perp}}^{n}\psi_{M-n}(r_e, r_s)\psi_1(r_e, r_s), \label{eq: func_s4}
\end{align}    

Note that the test functions provided above are not bounded, and therefore do not belong to the function space $\D$. To address this issue, we apply a standard truncation technique, as described in \citet{GuanChenDai2024}, which is omitted here for brevity.

\subsection{Proof of Base Step When $n=0$} \label{sec: base}

In this section, we aim to demonstrate that Statements (S1)-(S4) hold when $n = 0$. Clearly, (S1) and (S2) are trivially satisfied for $n = 0$ since $C_{0}=1$ and $D_{0}=J+1$. Throughout the rest of the paper, we use the shorthand notation $\psi$ for $\psi(R_e^\uu, R_s^\uu)$ under $\pi$ and for $\psi(R_{-,e}^\uu, R_{-,s}^\uu)$ under Palm measures, when the context is clear.

To prove Statement (S3) when $n = 0$, we employ the test function:
\begin{equation} \label{eq: func_s3_base}
	f_{0,E}\left(x\right) = \psi_{M+1}(r_e, r_s)=r_e^{M+1} + \sum_{j\in \J} r_{s,j}^{M+1}.
\end{equation}
Applying the test function $f_{0,E}$ to the BAR \eqref{eq:BAR}, the left-hand side of \eqref{eq:BAR} becomes
\begin{equation*}
	-\E_{\pi}\left[ \A f\left(X^\uu\right) \right]  =(M+1) \E_{\pi}\left[\left( \left( R_e^\uu \right)^{M} + \sum_{j\in\J} \left( R_{s,j}^\uu \right)^{M} \I{Z^\uu_j>0} \right) \right]. 
\end{equation*}
The right-hand side of the BAR \eqref{eq:BAR} corresponding to the external arrival is
\begin{align*}
	\E_{e} \left[ f_{0,E}\left(X_+^\uu\right)-f_{0,E}\left(X_-^\uu\right) \right]  &=\E_{e} \left[ \psi_{M+1}\left(R_{+,e}^\uu, R_{+,s}^\uu\right) - \psi_{M+1}\left(R_{-,e}^\uu, R_{-,s}^\uu\right) \right] \\
	& = \E_{e} \left[ \left(  \frac{T_{e}}{\alpha^\uu} \right)^{M+1} \right] = \frac{\E \left[   T_{e}^{M+1}  \right] }{(\alpha^\uu)^{M+1}},
\end{align*}
where the last equality follows from Lemma \ref{lmm: independent}.
The right-hand side of the BAR \eqref{eq:BAR} corresponding to the service completion at station $\ell \in \J$ is
\begin{align*}
	\E_{s,\ell} \left[ f_{0,E}\left(X_+^\uu\right)-f_{0,E}\left(X_-^\uu\right) \right] &=\E_{s,\ell} \left[ \psi_{M+1}\left(R_{+,e}^\uu, R_{+,s}^\uu\right) - \psi_{M+1}\left(R_{-,e}^\uu, R_{-,s}^\uu\right) \right] \\
	& = \E_{s,\ell} \left[ \left(  \frac{T_{s,\ell}}{\mu_\ell} \right)^{M+1} \right] = \frac{\E \left[   T_{s,\ell}^{M+1}  \right] }{\mu_\ell^{M+1}},
\end{align*}
where the last equality follows from Lemma \ref{lmm: independent}.

Combining the above results, the BAR \eqref{eq:BAR} with the test function $f_{0,E}$ yields
\begin{align*}
	&(M+1) \E_{\pi}\left[\left( \left( R_e^\uu \right)^{M} + \sum_{j\in\J} \left( R_{s,j}^\uu \right)^{M} \I{Z^\uu_j>0} \right) \right] \\
	&\quad\leq  \alpha^\uu \frac{\E \left[   T_{e}^{M+1}  \right] }{\left( \alpha^\uu \right)^{M+1}} + \sum_{\ell\in \J} \lambda_{\ell}^\uu \frac{\E \left[   T_{s,\ell}^{M+1}  \right] }{\mu_\ell^{M+1}}
	\leq  \frac{\E \left[   T_{e}^{M+1}  \right] }{\mu_{\max}^{M}} + \sum_{\ell\in \J} \frac{\E \left[   T_{s,\ell}^{M+1}  \right] }{\mu_\ell^{M}},
\end{align*}
where the final inequality holds due $\alpha^\uu \geq \mu_{\max}$ for all $r\in (0,r_0)$ and $\lambda_{\ell}^\uu \leq \mu_\ell$ by the definition of the departure rate.

Following the model setting in \eqref{eq: R_s}, the residual service time \(R^{(r)}_{s,j}(t)\) is set to the service time of the next job to be processed at station \(j\) if \(Z_j^\uu(t) = 0\). Thus, conditioning on the event \(\{Z_j^{(r)}=0\}\) under \(\pi\), \(R_{s, j}^{(r)} \stackrel{d}{=} {T_{s, j}}/{\mu_j^{(r)}}\) for any $j\in \mathcal{J}$, we then have 
\begin{align*}
	\E_{\pi}\left[   \left( R_{s,j}^\uu \right)^{M} \I{Z^\uu_j=0}   \right] =  \E_{\pi}\left[    \frac{\E[T^{M}_{s,j}]}{\mu^{M}_j}  \I{Z^\uu_j=0}  \right] \leq   \frac{\E[T^{M+1}_{s,j}]^{\frac{M}{M+1}}}{\mu^{M}_j},
\end{align*} 
and hence, we have
\begin{align*}
	\E_{\pi}[\psi_M]=\E_{\pi}\left[   \left( R_e^\uu \right)^{M } + \sum_{j\in\J} \left( R_{s,j}^\uu \right)^{M }  \right]  \leq  \frac{\E \left[   T_{e}^{M+1}  \right] }{\mu_{\max}^{M}} + \sum_{\ell\in \J} \frac{\E \left[   T_{s,\ell}^{M+1}  \right] }{\mu_\ell^{M}} + \sum_{j\in\J}  \frac{\E[T^{M+1}_{s,j}]^{\frac{M }{M+1}}}{\mu^{M }_j},
\end{align*}
which completes the proof of (S3) when $n=0$.

To prove Statement (S4) when $n = 0$, we employ the test function:
\begin{equation} \label{eq: func_s4_base}
	f_{0,F}\left(x\right) = \psi_{M}(r_e, r_s)\psi_1(r_e, r_s).
\end{equation}
Applying the test function $f_{0,F}$ to the BAR \eqref{eq:BAR}, the left-hand side of \eqref{eq:BAR} becomes
\begin{align*}
	-\E_{\pi}\left[ \A f\left(X^\uu\right) \right] & =\E_{\pi}\left[   \psi_{M } \cdot \left(1 + \sum_{j\in \J}  \I{Z_j^\uu >0} \right) \right]\notag \\
	&\quad + M \E_{\pi}\left[  \psi_1\cdot \left( \left( R_e^\uu \right)^{M -1} + \sum_{j\in\J} \left( R_{s,j}^\uu \right)^{M -1} \I{Z^\uu_j>0} \right) \right]\notag \\
	&  \leq (J+1)\E_{\pi}\left[   \psi_{M } \right]+M\E_{\pi}\left[   \psi_{M -1} \psi_1  \right] \leq (2J + 4J^2M)E_0, 
\end{align*}
where the final inequality holds due to the base step of Statement (S3) and the fact that
\begin{equation} \label{eq: 2psi}
	\psi_{M -1}(r_e, r_s)\psi_1(r_e, r_s) \leq (J+1)r_{\max}^{M-1} \cdot (J+1) r_{\max} \leq 4J^2 r_{\max}^{M} \leq 4J^2 \psi_{M}(r_e, r_s),
\end{equation}
where $r_{\max} = \max\{\{r_e\}\cup\{r_{s,j}\}_{j\in \J}\}$ and $J\geq 1$.

The right-hand side of the BAR \eqref{eq:BAR} corresponding to the external arrival is
\begin{align*}
	\E_{e} \left[ f_{0,F}\left(X_+^\uu\right)-f_{0,F}\left(X_-^\uu\right) \right] 
	& = \E_{e} \left[ \left( \psi_{M}  + \left( \frac{T_e}{\alpha^\uu} \right)^M \right)\left( \psi_1 + \frac{T_e}{\alpha^\uu} \right) - \psi_{M} \psi_1  \right] \\
	& \geq  \E_{e} \left[ \psi_{M}   \frac{T_e}{\alpha^\uu}   \right] =  \frac{1}{\alpha^\uu}   \E_{e} \left[ \psi_{M}   \right] ,
\end{align*}
where the inequality holds by neglecting the terms associated with $(T_e/\alpha^\uu)^M$, and the last equality follows from Lemma \ref{lmm: independent} and $\E  \left[ T_e \right] = 1$.

The right-hand side of the BAR \eqref{eq:BAR} corresponding to the service completion at station $\ell \in \J$ is
\begin{align*}
	\E_{s,\ell} \left[ f_{0,F}\left(X_+^\uu\right)-f_{0,F}\left(X_-^\uu\right) \right] 
	& = \E_{s,\ell} \left[ \left( \psi_{M}  + \left( \frac{T_{s,\ell}}{\mu_\ell} \right)^M \right)\left( \psi_1  + \frac{T_{s,\ell}}{\mu_\ell} \right) - \psi_{M} \psi_1  \right] \\
	& \geq  \E_{s,\ell} \left[ \psi_{M}   \frac{T_{s,\ell}}{\mu_\ell}   \right] =  \frac{1}{\mu_\ell}   \E_{s,\ell} \left[ \psi_{M}   \right] ,
\end{align*}
where the inequality holds by neglecting the terms associated with $(T_{s,\ell}/\mu_\ell)^M$, and the last equality follows from Lemma \ref{lmm: independent} and $\E  \left[ T_{s,\ell} \right] = 1$.

Combining the above results, the BAR \eqref{eq:BAR} with the test function $f_{0,F}$ yields
\begin{align*}
	(2J + 4J^2M)E_0 &\geq  \E_{e} \left[ \psi_{M}   \right] + \sum_{\ell\in \J} \frac{\lambda_\ell^\uu}{\mu_\ell}   \E_{s,\ell} \left[ \psi_{M}   \right] \geq \E_{e} \left[ \psi_{M}   \right] + \sum_{\ell\in \J} \frac{\mu_\ell-r}{\mu_\ell}   \E_{s,\ell} \left[ \psi_{M}   \right],
\end{align*}
where the last inequality follows from the property of the departure rate:
\begin{equation} \label{eq: lambda}
	\mu_\ell - r \leq \lambda_\ell^\uu \leq \mu_{\ell}.
\end{equation}
By definition, \(\lambda_\ell^\uu \leq \mu_\ell\) holds. Using the test function $f(x)=\sum_{j\in \J} z_j$, BAR \eqref{eq:BAR} indicates \(\lambda_\Sigma^\uu = \alpha^\uu\), and hence, we obtain \( r = \mu_\Sigma  - \alpha^\uu = \sum_{\ell \in \J} (\mu_\ell - \lambda_\ell^\uu) \). Since each component satisfies \(\mu_\ell - \lambda_\ell^\uu \geq 0\), we conclude \(\mu_\ell - \lambda_\ell^\uu \leq r\), completing the proof of \eqref{eq: lambda}. 

Therefore, we have
\begin{align*}
	\E_{e} \left[ \psi_{M}   \right] + \sum_{\ell\in \J}   \E_{s,\ell} \left[ \psi_{M}   \right] \leq \frac{(2J + 4J^2M)E_0}{\min_{\ell \in \J, r\in (0,r_0)}\frac{\mu_\ell - r}{ \mu_{\ell}}} \leq \frac{(2J + 4J^2M)E_0}{\frac{\mu_{\min}-r_0}{\mu_{\max}}},
\end{align*}
which completes the proof of (S4) when $n=0$.

\subsection{Proof of (S1)} \label{sec: proof of S1}
Assuming that Statements (S1)-(S4) hold for $0,1,\ldots,n-1$ according to the induction hypotheses, we proceed to prove Statement (S1) for the induction step $n$. 

The proof of (S1)  depends highly on the properties of the functions $\eta_{n,e}$ and $\eta_{n,s,j}$ in Lemma~\ref{lmm: prop_eta}, which is proved in Appendix \ref{sec: lemma}. To clarify the proof, we use $y^\uu = O(z^\uu)$ and $y^\uu = \Theta(z^\uu)$ to denote that $y^\uu/z^\uu$ is uniformly upper bounded or two-sided bounded by a constant over all $r\in (0,r_0)$, respectively. Note that this differs from the standard asymptotic notation. Specifically, there exists a constant $C>0$ such that 
\begin{equation*}
	y^\uu = O(z^\uu) \Leftrightarrow \sup_{r\in (0,r_0)} \frac{y^\uu}{z^\uu} \leq C, \quad y^\uu = \Theta(z^\uu) \Leftrightarrow \sup_{r\in (0,r_0)} \abs{\frac{y^\uu}{z^\uu}} \leq C.
\end{equation*}

We denote \(\delta_j\) as the jump increment of \(z_{\perp}\) due to an external arrival at station \(j\):
\begin{align}
	\delta_j &\equiv z_{+,\perp} - z_{-,\perp} = \left[ \left( z_- + \e{\ell} \right) - \left( \bar z_- + \frac{1}{J}  \right)\cdot e \right]  - \left( z_- - \bar z_- \cdot e \right) = \e{\ell} - e/J. \label{eq: delta_ell}
\end{align}
The jump increment of \(z_{\perp}\) due to a service completion at station \(j\) is then \(-\delta_j\). The following lemma demonstrates the properties of the functions \(\eta_{n,e}\) and \(\eta_{n,s,j}\), leading to the Palm terms in BAR \eqref{eq:BAR} having a polynomial in \(\norm{z_\perp}\) of order up to \(n-1\) and the proof is provided in the Appendix \ref{sec: lemma}.

\begin{lemma} \label{lmm: prop_eta}
	For each $n\geq 1$, the functions $\eta_{n,e}$ and $\eta_{n,s,j}$ satisfy the following properties:
	\begin{align}
		\label{eq: prop_e}
		&\sum_{j\in \J}\frac{1}{n+1}\norm{z_{\perp}+\delta_j}^{n+1}p_{j\mid z}  - \frac{1}{n+1}\norm{z_{\perp}}^{n+1} - \eta_{n,e}(z_{\perp}) = O(\norm{z_{\perp}}^{n-1}),\\
		\label{eq: prop_s}
		&\frac{1}{n+1}\norm{z_{\perp}-\delta_j}^{n+1} - \frac{1}{n+1}\norm{z_{\perp}}^{n+1} + \eta_{n,s,j}(z_{\perp}) = O(\norm{z_{\perp}}^{n-1}),\\
	&\eta_{n,e}(z_{\perp} + \delta) - \eta_{n,e}(z_{\perp}) = \Theta(\norm{z_{\perp}}^{n-1}), \text{ for all $\norm{\delta}\leq 1$} \label{eq: prop_e2}\\
	&\eta_{n,s,j}(z_{\perp} - \delta) - \eta_{n,s,j}(z_{\perp}) = \Theta(\norm{z_{\perp}}^{n-1}), \text{ for all $\norm{\delta}\leq 1$} \label{eq: prop_s2}\\
	\label{eq: coeffS1}
	&-\A f_n(x) = -\eta_{n,e}(z_{\perp}) \cdot \alpha^\uu  + \sum_{j\in \J} \eta_{n,s,j}(z_{\perp}) \cdot \mu_j \I{z_j>0} \geq b\norm{z_{\perp}}^{n},
\end{align}
with $b = \mu_{\min}/2\sqrt{J}$ for JSQ policy and $b = \mu_\Sigma p_{\min}^2/(2\sqrt{J})$ for Po2 policy.
\end{lemma}

Applying the test function $f_n$ in \eqref{eq: func_s1} to BAR \eqref{eq:BAR} and utilizing \eqref{eq: coeffS1}, the left-hand side of \eqref{eq:BAR} becomes
\begin{equation}
	-\E_{\pi}\left[\mathcal{A}f(X^\uu)\right]\geq b \E_{\pi}[  \norm{Z_{\perp}^\uu}^n ]  . \label{eq: LHS res}
\end{equation}

On the right-hand side of the BAR \eqref{eq:BAR} corresponding to the external arrival, the increment conditional on the pre-jump state $x_-=(z_-,r_{-,e}, r_{-,s})$ is
\begin{align*}
	&\E\left[ f_n(x_+)-f_n(x_-) \right] = \sum_{\ell \in \J}  \E\left[ f_n(x_- + \Delta_{e,\ell}) \right]p_{\ell\mid z} - f(x_-) \notag \\
	& = \sum_{\ell \in \J} \Big[ \frac{1}{n+1}\Norm{z_{-,\perp} + \delta_\ell}^{n+1} - \eta_{n,e}(z_{-,\perp} + \delta_\ell) \cdot \E[T_e] + \sum_{j\in \J} \eta_{n,s,j}(z_{-,\perp} + \delta_\ell) \cdot \mu_j r_{s,j}\Big]p_{\ell\mid z} - f(x_-) \notag \\
	& = \sum_{\ell \in \J}  \frac{1}{n+1}\Norm{z_{-,\perp} + \delta_\ell}^{n+1}p_{\ell\mid z}  - \frac{1}{n+1}\Norm{z_{-,\perp}}^{n+1} - \eta_{n,e}(z_{-,\perp})  \notag \\
	&\quad + \sum_{\ell \in \J}  \left[ \eta_{n,e}(z_{-,\perp}) -\eta_{n,e}(z_{-,\perp} + \delta_\ell)\right]p_{\ell\mid z} + \sum_{\ell \in \J} \sum_{j\in \J} \left[ \eta_{n,s,j}(z_{-,\perp} + \delta_\ell)  - \eta_{n,s,j}(z_{-,\perp})\right] \mu_j r_{s,j}p_{\ell\mid z} \notag \\
	& = O(\norm{z_{-,\perp}}^{n-1}) + \sum_{j\in \J} O(\norm{z_{-,\perp}}^{n-1})r_{-,s,j}= O(\norm{z_{-,\perp}}^{n-1}) + O(\norm{z_{-,\perp}}^{n-1})\psi_1, 
\end{align*} 
where the second equality holds since $T_e$ and $X^\uu_-$ are independent in Lemma \ref{lmm: independent}, the third equality holds due to $\E[T_e]=1$, and the last equality follows from \eqref{eq: prop_e} and \eqref{eq: prop_e2} in Lemma~\ref{lmm: prop_eta}.
Taking the expectation under $\mathbb{P}_e$, the right-hand side of \eqref{eq:BAR} corresponding to the external arrival becomes
\begin{align}
	&\E_{e}\left[ f_n(X_+^\uu)-f_n(X_-^\uu)  \right]=\E_{e}\left[ \E\left[ f_n(x_+)-f_n(x_-) \mid X_-^\uu=x_-\right] \right] \notag \\
	&\quad = \E_{e}\left[ O(\norm{Z^\uu_{-,\perp}}^{n-1}) \right] +   \E_{e}\left[ O(\norm{Z^\uu_{-,\perp}}^{n-1})\psi_1 \right] \leq \Theta_1. \label{eq: RHS1 res}
\end{align}
Since Hölder's inequality indicates that $\E_{e}[ O(\norm{Z^\uu_{-,\perp}}^{n-1})\psi_1 ]\leq \E_{e}[ O(\norm{Z^\uu_{-,\perp}}^{n-1}) ]^{(M-n)/(M-n+1)}\cdot$ $\E_{e}[ O(\norm{Z^\uu_{-,\perp}}^{n-1})\psi_1^{M-n+1} ]^{1/(M-n+1)}$, the inequality in \eqref{eq: RHS1 res} follows from the induction hypotheses (S2) and (S4) for $n-1$, and \(\Theta_1\) is a positive constant independent of \(r\).

On the right-hand side of the BAR \eqref{eq:BAR} corresponding to the service completion at station $\ell \in \J$, the increment conditional on the pre-jump state $x_-=(z_-,r_{-,e}, r_{-,s})$ is
\begin{align*}
	&\E\left[ f_n(x_+)-f_n(x_-) \right] = \E\left[ f_n(x_- + \Delta_{s,\ell}) \right] - f(x_-) \\
	& =  \frac{1}{n+1}\Norm{z_{-,\perp} - \delta_\ell}^{n+1}
	 - \eta_{n,e}(z_{-,\perp} - \delta_\ell) \cdot \alpha^\uu r_{-,e} + \sum_{j\in \J, j\neq \ell} \eta_{n,s,j}(z_{-,\perp} - \delta_\ell) \cdot \mu_j r_{-,s,j} \\
	&   + \eta_{n,s,\ell}(z_{-,\perp} - \delta_\ell) \cdot \E[T_{s,\ell}]  - \Big[ \frac{1}{n+1}\Norm{z_{-,\perp}}^{n+1} - \eta_{n, e}(z_{-,\perp}) \cdot \alpha^\uu r_{-,e} + \sum_{j\in \J, j\neq \ell} \eta_{n,s,j}(z_{-,\perp}) \cdot \mu_j r_{-,s,j} \Big] \\
	& = \frac{1}{n+1}\Norm{z_{-,\perp} - \delta_\ell}^{n+1} - \frac{1}{n+1}\Norm{z_{-,\perp}}^{n+1} + \eta_{n,s,\ell}(z_{-,\perp}) + \eta_{n,s,\ell}(z_{-,\perp} - \delta_\ell) - \eta_{n,s,\ell}(z_{-,\perp})  \\
	&\quad+ \left( \eta_{n,e}(z_{-,\perp}) - \eta_{n,e}(z_{-,\perp} - \delta_\ell) \right) \cdot \alpha^\uu r_{-,e}  + \sum_{j\in \J, j\neq \ell} \left(  \eta_{n,s,j}(z_{-,\perp} - \delta_\ell) - \eta_{n,s,j}(z_{-,\perp}) \right) \cdot \mu_j r_{-,s,j} \\
	&= O(\norm{z_{-,\perp}}^{n-1}) + O(\norm{z_{-,\perp}}^{n-1})r_{-,e} + \sum_{j\in \J, j\neq \ell} O(\norm{z_{-,\perp}}^{n-1})r_{-,s,j}= O(\norm{z_{-,\perp}}^{n-1}) + O(\norm{z_{-,\perp}}^{n-1})\psi_1, 
\end{align*}
where  the second equality follows from the fact that $T_{s,\ell}$ and $X^\uu_-$ are independent in Lemma \ref{lmm: independent}, the third equality holds due to $\E[T_{s,\ell}]=1$, and the last equality follows from \eqref{eq: prop_s} and \eqref{eq: prop_s2} in Lemma~\ref{lmm: prop_eta}. Taking the expectation under $\mathbb{P}_{s,\ell}$, the right-hand side of \eqref{eq:BAR} that corresponds to the service completion at station $\ell$ becomes
\begin{align}
	&\E_{s,\ell}\left[ f_n(X_+^\uu)-f_n(X_-^\uu)  \right]=\E_{s,\ell}\left[ \E\left[ f_n(x_+)-f_n(x_-) \mid X_-^\uu=x_-\right] \right] \notag \\
	&\quad = \E_{s,\ell}\left[ O(\norm{Z^\uu_{-,\perp}}^{n-1}) \right] +   \E_{s,\ell}\left[ O(\norm{Z^\uu_{-,\perp}}^{n-1})\psi_1 \right] \leq \Theta_2, \label{eq: RHS2 res}
\end{align}
where the inequality follows from the induction hypotheses (S2) and (S4) for $n-1$, and \(\Theta_2\) is a positive constant independent of \(r\).

In summary, combining \eqref{eq: LHS res}, \eqref{eq: RHS1 res} and \eqref{eq: RHS2 res}, the BAR \eqref{eq:BAR} implies that
\begin{equation*}
	b \E_{\pi}\left[  \Norm{Z_{\perp}^\uu}^n \right] \leq \alpha^\uu \Theta_1 + \sum_{\ell \in \J} \lambda^\uu \Theta_2 = \left( \Theta_1 +   \Theta_2 \right) \alpha^\uu \leq  \left( \Theta_1 +   \Theta_2 \right) \mu_\Sigma, 
\end{equation*}
which completes the proof of (S1) for the induction step $n$.

\subsection{Proof of (S2)} \label{sec: proofS2}
Based on the induction hypotheses, we assume that Statement (S1) holds for $0,1\ldots,n$, and Statements (S2)-(S4) are satisfied for $0,1,\ldots, n-1$. We will now demonstrate the validity of Statement (S2) for the induction step $n$ by substituting $f_{n,D}$ as defined in \eqref{eq: func_s2} into the BAR \eqref{eq:BAR}.

On the left-hand side of the BAR \eqref{eq:BAR}, we have
\begin{equation*}
	-\A f_{n,D}(x) = \norm{z_{\perp}}^{n}\left( 1 + \sum_{j\in \J} \I{z_j>0} \right) \leq (J+1)\norm{z_{\perp}}^{n},
\end{equation*}
and hence, the left-hand side of the BAR \eqref{eq:BAR} becomes
\begin{equation}
	-\E_{\pi}\left[ \A f_{n,D}\left(X^\uu\right) \right]   \leq (J+1)\E_{\pi}\left[ \Norm{Z_{\perp}^\uu}^{n}  \right] \leq (J+1)C_n, \label{eq: LHS res2}
\end{equation}
where the last inequality follows from the induction hypothesis (S1) for $n$.

On the right-hand side of the BAR \eqref{eq:BAR} corresponding to the external arrival, conditioning on the pre-jump state $x_-=(z_-,r_{-,e}, r_{-,s})$, we have
\begin{align*}
	\E[f_{n,D}(x_+)-f_{n,D}(x_-)] &= \sum_{\ell\in \J} \E\left[ f_{n,D}(x_- + \Delta_{e,\ell}) \right]p_{\ell\mid z} - f_{n,D}(x_-) \notag \\
	&  = \left( \psi_1 + \frac{\E[T_e]}{\alpha^\uu} \right) \sum_{\ell\in \J} \Norm{z_{-,\perp} + \delta_\ell}^{n}  p_{\ell\mid z} - \Norm{z_{-,\perp}}^{n} \psi_1 \notag \\
	&  =  \left( \psi_1 + \frac{1}{\alpha^\uu} \right)\sum_{\ell\in \J} \left(\Norm{z_{-,\perp} + \delta_\ell}^{n} - \Norm{z_{-,\perp}}^{n} \right) p_{\ell\mid z} + \Norm{z_{-,\perp}}^{n}\frac{1}{\alpha^\uu} \notag \\
	& = \left( \psi_1 + \frac{1}{\alpha^\uu} \right) \Theta(\norm{z_{-,\perp}}^{n-1}) + \Norm{z_{-,\perp}}^{n}\frac{1}{\alpha^\uu}, 
\end{align*}
where the second equality follows from the fact that $T_e$ and $X^\uu_-$ are independent in Lemma \ref{lmm: independent}, the third equality holds by $\E[T_e]=1$, and the last equality follows from \eqref{eq: prop_e2}. Taking the expectation under $\mathbb{P}_e$, the right-hand side of \eqref{eq:BAR} that corresponds to the external arrival becomes
\begin{align}
	&\E_{e}\left[ f_{n,D}(X^\uu_+)-f_{n,D}(X^\uu_-) \right] = \E_e\left[ \E\left[ f_{n,D}(x_+)-f_{n,D}(x_-) \mid X_-^\uu =x_-\right] \right] \notag \\
	&\quad = \E_e\left[ \left( \psi_1 + \frac{1}{\alpha^\uu} \right) \Theta(\norm{Z_{-,\perp}^\uu}^{n-1})\right] +\frac{1}{\alpha^\uu} \E_e\left[ \Norm{Z_{-,\perp}^\uu}^{n} \right] \geq -\Theta_3 + \frac{1}{\alpha^\uu} \E_e\left[ \Norm{Z_{-,\perp}^\uu}^{n} \right], \label{eq: RHS1 res2}
\end{align}
where the last inequality holds by the induction hypotheses (S2) and (S4) for $n-1$, and \(\Theta_3\) is a positive constant independent of \(r\).

On the right-hand side of the BAR \eqref{eq:BAR} corresponding to the service completion at station $\ell$, conditional on the pre-jump state $x_-=(z_-,r_{-,e}, r_{-,s})$, the increment of $f_{n,D}$ is
\begin{align*}
	&\E[f_{n,D}(x_+)-f_{n,D}(x_-)] =  \E[f_{n,D}(x_- + \Delta_{s,\ell})-f_{n,D}(x_-)] = \Norm{z_{-,\perp} - \delta_\ell}^{n}\left( \psi_1 + \frac{\E[T_{s,\ell}]}{\mu_\ell} \right) -  \Norm{z_{-,\perp}}^{n}  \psi_1  \\
	& \quad =    \left( \Norm{z_{-,\perp} - \delta_\ell}^{n} -  \Norm{z_{-,\perp}}^{n} \right) \left( \psi_1 + \frac{1}{\mu_\ell} \right) +  \Norm{z_{-,\perp}}^{n}  \frac{1}{\mu_\ell}  = \Theta(\norm{z_{-,\perp}}^{n-1})\left( \psi_1 + \frac{1}{\mu_\ell} \right) +  \Norm{z_{-,\perp}}^{n}  \frac{1}{\mu_\ell},
\end{align*}
where the second equality follows from the fact that $T_{s,\ell}$ and $X^\uu_-$ are independent in Lemma \ref{lmm: independent}, the third equality holds by $\E[T_{s,\ell}]=1$, and the last equality follows from \eqref{eq: prop_s2}. Taking the expectation under $\mathbb{P}_{s,\ell}$, the right-hand side of \eqref{eq:BAR} that corresponds to the service completion at station $\ell$ becomes
\begin{align}
	&\E_{s,\ell}\left[ f_{n,D}(X^\uu_+)-f_{n,D}(X^\uu_-) \right] = \E_{s,\ell}\left[ \E\left[ f_{n,D}(x_+)-f_{n,D}(x_-) \mid X_-^\uu =x_-\right] \right] \notag \\
	&\quad = \E_{s,\ell}\left[ O(\norm{Z_{-,\perp}^\uu}^{n-1})\left( \psi_1 + \frac{1}{\mu_\ell} \right) \right] +  \E_{s,\ell}\left[ \Norm{Z_{-,\perp}^\uu}^{n}  \frac{1}{\mu_\ell} \right] \geq -\Theta_4 + \frac{1}{\mu_\ell}  \E_{s,\ell}\left[ \Norm{Z_{-,\perp}^\uu}^{n} \right], \label{eq: RHS2 res2}
\end{align}
where the last inequality holds by the induction hypotheses (S2) and (S4) for $n-1$, and \(\Theta_4\) is a positive constant independent of \(r\).

In summary, combining \eqref{eq: LHS res2}, \eqref{eq: RHS1 res2} and \eqref{eq: RHS2 res2}, the BAR \eqref{eq:BAR} implies that
\begin{align*}
	(J+1)C_n \geq -\alpha^\uu \Theta_3 + \E_{e}\left[ \Norm{Z_{-,\perp}^\uu}^{n}   \right] + \sum_{\ell\in\J} \lambda^\uu_\ell \left( -\Theta_4 + \frac{1}{ \mu_{\ell}}\E_{s,\ell}\left[\Norm{Z_{-,\perp}^\uu}^{n}  \right] \right)\\
	\geq -\alpha^\uu \left( \Theta_3 + \Theta_4 \right) + \E_{e}\left[ \Norm{Z_{-,\perp}^\uu}^{n}   \right] + \sum_{\ell\in\J}   \frac{\mu_\ell - r}{ \mu_{\ell}}\E_{s,\ell}\left[\Norm{Z_{-,\perp}^\uu}^{n}  \right],
\end{align*}
where the last inequality follows from $\lambda^\uu_\Sigma = \alpha^\uu$ and \eqref{eq: lambda}.

Therefore, we have
\begin{align*}
	\E_{e}\left[ \Norm{Z_{-,\perp}^\uu}^{n}   \right] + \sum_{\ell\in\J}  \E_{s,\ell}\left[\Norm{Z_{-,\perp}^\uu}^{n}  \right] &\leq \frac{(J+1)C_n +\alpha^\uu \left( \Theta_3 + \Theta_4 \right) }{\min_{\ell \in \J, r\in (0,r_0)}\frac{\mu_\ell - r}{ \mu_{\ell}}}  \leq \frac{(J+1)C_n + \left( \Theta_3 + \Theta_4 \right) \mu_{\Sigma} }{\frac{\mu_{\min}-r_0}{\mu_{\max}}},
\end{align*}
which completes the proof of (S2) for the induction step $n$. 

\subsection{Proof of (S3)}
Building upon the induction hypotheses, we assume that Statements (S1) and (S2) hold for $0,1,\ldots, n$, while Statements (S3) and (S4) are satisfied for $0,1,\ldots, n-1$. In this section, we aim to prove Statement (S3) for the induction step $n$. Note that when $n = M$, Statement (S3) holds trivially due to Statement (S1). Hence, in the subsequent analysis, we will solely focus on the cases when $1 \leq n < M$.

Plugging $f_{n,E}$ in \eqref{eq: func_s3} into BAR \eqref{eq:BAR}, the left-hand side becomes
\begin{equation}
	-\E_{\pi}\left[ \A f\left(X^\uu\right) \right]  =(M-n+1) \E_{\pi}\left[ \Norm{Z_{\perp}^\uu}^{n}\left( \left( R_e^\uu \right)^{M-n} + \sum_{j\in\J} \left( R_{s,j}^\uu \right)^{M-n} \I{Z^\uu_j>0} \right) \right]. \label{eq: LHS res3}
\end{equation}

On the right-hand side of the BAR \eqref{eq:BAR} corresponding to the external arrival, conditional on the pre-jump state $x_-=(z_-,r_{-,e}, r_{-,s})$, we have
\begin{align*}
	&\E\left[ f_{n,E}(x_+)-f_{n,E}(x_-) \right] = \sum_{\ell\in \J} \E\left[ f_{n,E}(x_- + \Delta_{e,\ell}) \right]p_{\ell\mid z} - f_{n,E}(x_-) \notag \\
	& \quad = \left( \psi_{M-n+1} + \E\left[\left( \frac{T_e}{\alpha^\uu} \right)^{M-n+1}\right] \right) \sum_{\ell\in \J}  \Norm{z_{-,\perp} + \delta_\ell}^{n}p_{\ell\mid z} - \Norm{z_{-,\perp}}^{n} \psi_{M-n+1}  \notag \\
	& \quad = \left( \psi_{M-n+1} + \frac{\E[T_e^{M-n+1}]}{(\alpha^\uu)^{M-n+1}} \right) \sum_{\ell\in \J} \left( \Norm{z_{-,\perp} + \delta_\ell}^{n} - \Norm{z_{-,\perp}}^{n} \right) p_{\ell\mid z} + \Norm{z_{-,\perp}}^{n} \frac{\E[T_e^{M-n+1}]}{(\alpha^\uu)^{M-n+1}} \notag \\
	& \quad= \left( \psi_{M-n+1} + \frac{\E[T_e^{M-n+1}]}{(\alpha^\uu)^{M-n+1}} \right) O(\norm{z_{-,\perp}}^{n-1}) + \Norm{z_{-,\perp}}^{n} \frac{\E[T_e^{M-n+1}]}{(\alpha^\uu)^{M-n+1}},
\end{align*}
where the second equality follows from the fact that $T_e$ and $X^\uu_-$ are independent in Lemma \ref{lmm: independent}, and the last equality follows from \eqref{eq: prop_e2}. Taking the expectation under $\mathbb{P}_e$, the right-hand side of \eqref{eq:BAR} that corresponds to the external arrival becomes
\begin{align}
	&\E_{e}\left[ f_{n,E}(X^\uu_+)-f_{n,E}(X^\uu_-) \right] = \E_{e}\left[ \E\left[ f_{n,E}(x_+)-f_{n,E}(x_-) \mid X_-^\uu =x_-\right] \right] \notag \\
	&\quad = \E_{e}\left[ \left( \psi_{M-n+1} + \frac{\E[T_e^{M-n+1}]}{(\alpha^\uu)^{M-n+1}} \right) \Theta(\norm{Z_{-,\perp}^\uu}^{n-1}) \right] + \frac{\E[T_e^{M-n+1}]}{(\alpha^\uu)^{M-n+1}}\E_{e}\left[ \Norm{Z_{-,\perp}^\uu}^{n}  \right] \leq \Theta_5  \label{eq: RHS1 res3}
\end{align}
where the last inequality holds by the induction hypotheses (S2) for $n-1$ and $n$ together with (S4) for $n-1$, and \(\Theta_5\) is a positive constant independent of \(r\).

On the right-hand side of the BAR \eqref{eq:BAR} corresponding to the service completion at station $\ell$, conditional on the pre-jump state $x_-=(z_-,r_{-,e}, r_{-,s})$, the increment of $f_{n,E}$ is
\begin{align*}
	&\E\left[ f_{n,E}(x_+)-f_{n,E}(x_-) \right] = \E\left[ f_{n,E}(x_- + \Delta_{s,\ell}) \right] - f_{n,E}(x_-) \\
	& = \left( \psi_{M-n+1} + \frac{\E[T_{s,\ell}^{M-n+1}]}{\mu_\ell^{M-n+1}} \right) \Norm{z_{-,\perp} - \delta_\ell}^{n} - \Norm{z_{-,\perp}}^{n} \psi_{M-n+1} \\
	& = \left( \psi_{M-n+1} + \frac{\E[T_{s,\ell}^{M-n+1}]}{\mu_\ell^{M-n+1}} \right) \left( \Norm{z_{-,\perp} - \delta_\ell}^{n} - \Norm{z_{-,\perp}}^{n} \right) + \Norm{z_{-,\perp}}^{n} \frac{\E[T_{s,\ell}^{M-n+1}]}{\mu_\ell^{M-n+1}} \\
	& = \left( \psi_{M-n+1} + \frac{\E[T_{s,\ell}^{M-n+1}]}{\mu_\ell^{M-n+1}} \right) \Theta(\norm{z_{-,\perp}}^{n-1}) + \Norm{z_{-,\perp}}^{n} \frac{\E[T_{s,\ell}^{M-n+1}]}{\mu_\ell^{M-n+1}},
\end{align*}
where the second equality follows from the fact that $T_{s,\ell}$ and $X^\uu_-$ are independent in Lemma \ref{lmm: independent}, and the last equality follows from \eqref{eq: prop_s2}. Taking the expectation under $\mathbb{P}_{s,\ell}$, the right-hand side of \eqref{eq:BAR} that corresponds to the service completion at station $\ell$ becomes
\begin{align}
	&\E_{s,\ell}\left[ f_{n,E}(X^\uu_+)-f_{n,E}(X^\uu_-) \right] = \E_{s,\ell}\left[ \E\left[ f_{n,E}(x_+)-f_{n,E}(x_-) \mid X_-^\uu =x_-\right] \right] \notag \\
	&\quad = \E_{s,\ell}\left[ \left( \psi_{M-n+1} + \frac{\E[T_{s,\ell}^{M-n+1}]}{\mu_\ell^{M-n+1}} \right) \Theta(\norm{Z_{-,\perp}^\uu}^{n-1}) \right] + \frac{\E[T_{s,\ell}^{M-n+1}]}{\mu_\ell^{M-n+1}}\E_{s,\ell}\left[ \Norm{Z_{-,\perp}^\uu}^{n}  \right] \leq \Theta_6, \label{eq: RHS2 res3}
\end{align}
where the last inequality holds by the induction hypotheses (S2) for $n-1$ and $n$ together with (S4) for $n-1$, and \(\Theta_6\) is a positive constant independent of \(r\).

In summary, combining \eqref{eq: LHS res3}, \eqref{eq: RHS1 res3} and \eqref{eq: RHS2 res3}, the BAR \eqref{eq:BAR} implies that
\begin{align*}
	\E_{\pi}\left[ \Norm{Z_{\perp}^\uu}^{n}\left( \left( R_e^\uu \right)^{M-n} + \sum_{j\in\J} \left( R_{s,j}^\uu \right)^{M-n} \I{Z^\uu_j>0} \right) \right]&\leq \frac{\alpha^\uu \Theta_5 + \sum_{\ell\in\J} \lambda^\uu_\ell \Theta_6}{M-n+1}  \leq \frac{ \left( \Theta_5 +   \Theta_6 \right)\mu_{\Sigma}}{M-n+1} .
\end{align*}

Following the model setting in \eqref{eq: R_s}, the residual service time \(R^{(r)}_{s,j}(t)\) is set to the service time of the next job to be processed at station \(j\) if \(Z_j^\uu(t) = 0\). Thus, \(R^{(r)}_{s,j}(t)\) is independent of the $\sigma$-algebra $\mathcal{F}_{t-}^{X^\uu}$, and hence is independent of $Z^\uu(t)$. Conditional on the event \(\{Z_j^{(r)}=0\}\) under \(\pi\), \(R_{s, j}^{(r)}\) has the same distribution as \({T_{s, j}}/{\mu_j}\) and is independent of $Z^\uu$.  Therefore, we have
\begin{align*}
	\E_{\pi}\left[ \Norm{Z_{\perp}^\uu}^{n}  \left( R_{s,j}^\uu \right)^{M-n} \I{Z^\uu_j=0}   \right] =  \E_{\pi}\left[ \Norm{Z_{\perp}^\uu}^{n}  \frac{\E[T^{M-n}_{s,j}]}{\mu^{M-n}_j}  \I{Z^\uu_j=0}  \right] \leq C_{n} \frac{\E[T^{M+1}_{s,j}]^{\frac{M-n}{M+1}}}{\mu^{M-n}_j},
\end{align*}
where the last inequality holds by the induction hypothesis (S1) for $n$, and hence,
\begin{align*}
	\E_{\pi}\left[ \Norm{Z_{\perp}^\uu}^{n}\psi_{M-n}\right]&\leq \frac{ \left( \Theta_5 +   \Theta_6 \right)\sum_{j\in \J}\mu_{j}}{M-n+1}  + \sum_{j\in\J}C_{n} \frac{\E[T^{M+1}_{s,j}]^{\frac{M-n}{M+1}}}{\mu^{M-n}_j},
\end{align*}
which completes the proof of (S3) for the induction step $n$.

\subsection{Proof of (S4)} \label{sec:proofS4}
Given the induction hypotheses, we assume that Statements (S1)-(S3) hold for $0,1,\ldots,n$, and Statement (S4) is satisfied for $0,1,\ldots,n-1$. We will now prove Statement (S4) for the induction step $n$. However, if $n = M$, Statement (S4) is trivially satisfied due to Statement (S2). Therefore, in the following analysis, we will focus on the cases where $1 \leq n < M$.

On the left side of the BAR \eqref{eq:BAR}, we have
\begin{align*}
	-\A f_{n,F}(x) &= \Norm{z_{\perp}}^{n}\left(1 + \sum_{j\in \J}  \I{z_j >0} \right) \psi_{M-n}  + (M-n) \Norm{z_{\perp}}^{n} \left( r_e^{M-n-1} + \sum_{j\in\J} r_{s,j}^{M-n-1} \I{z_j>0} \right) \psi_1 \\
	&\leq (J+1) \Norm{z_{\perp}}^{n} \psi_{M-n} + (M-n)  \Norm{z_{\perp}}^{n} \psi_1 \psi_{M-n-1} \leq \left( 2J + 4J^2M  \right) \Norm{z_{\perp}}^{n} \psi_{M-n},
\end{align*}
where the last inequality holds due to the property stated in  \eqref{eq: 2psi}. Hence, the left-hand side of the BAR \eqref{eq:BAR} becomes 
\begin{align}
	&-\E_{\pi}\left[ \A f\left(X^\uu\right) \right]  \leq (2J + 4J^2M)\E_{\pi}\left[ \Norm{Z_{\perp}^\uu}^{n} \psi_{M-n} \right] \leq (2J + 4J^2M )E_n, \label{eq: LHS res4}
\end{align}
where the final inequality follows from the induction hypothesis (S3) for $n$.

On the right-hand side of the BAR \eqref{eq:BAR} corresponding to the external arrival, conditional on the pre-jump state $x_-=(z_-,r_{-,e}, r_{-,s})$, we have
\begin{align*}
	&\E\left[ f_{n,F}(x_+)-f_{n,F}(x_-) \right] = \sum_{\ell\in \J} \E\left[ f_{n,F}(x_- + \Delta_{e,\ell}) \right]p_{\ell\mid z} - f_{n,F}(x_-) \notag \\
	& \quad = \left( \psi_{M-n} + \E\left[\left( {T_e}/{\alpha^\uu} \right)^{M-n} \right] \right) \left( \psi_1 + \E\left[ {T_e}/{\alpha^\uu} \right] \right) \sum_{\ell\in \J}  \Norm{z_{-,\perp} + \delta_\ell}^{n}p_{\ell\mid z} - \Norm{z_{-,\perp}}^{n} \psi_{M-n}\psi_1 \notag \\
	& \quad \geq \psi_{M-n} \left( \psi_1 + {1}/{\alpha^\uu} \right) \sum_{\ell\in \J} \left( \Norm{z_{-,\perp} + \delta_\ell}^{n} - \Norm{z_{-,\perp}}^{n} \right) p_{\ell\mid z} + \Norm{z_{-,\perp}}^{n} \psi_{M-n} /{\alpha^\uu} \notag \\
	& \quad = \Theta(\norm{z_{-,\perp}}^{n-1}\psi_{M-n+1}) + \Norm{z_{-,\perp}}^{n} \psi_{M-n} /{\alpha^\uu}, 
\end{align*}
where the second equality follows from the fact that $T_e$ and $X^\uu_-$ are independent in Lemma \ref{lmm: independent}, and the last inequality follows from \eqref{eq: prop_e2}. Taking the expectation under $\mathbb{P}_e$, the right-hand side of \eqref{eq:BAR} that corresponds to the external arrival becomes
\begin{align}
	&\E_{e}\left[ f_{n,F}(X^\uu_+)-f_{n,F}(X^\uu_-) \right] = \E_{e}\left[ \E\left[ f_{n,F}(x_+)-f_{n,F}(x_-) \mid X_-^\uu =x_-\right] \right] \notag \\
	&~ = \E_{e}\left[ \Theta(\norm{Z_{-,\perp}^\uu}^{n-1}\psi_{M-n+1}) \right] +   \E_{e}\left[ \Norm{Z_{-,\perp}^\uu}^{n} \psi_{M-n} \right]/{\alpha^\uu} \geq -\Theta_7 +  \E_{e}\left[ \Norm{Z_{-,\perp}^\uu}^{n} \psi_{M-n} \right]/{\alpha^\uu}, \label{eq: RHS1 res4}
\end{align}
where the last inequality holds by the induction hypothesis (S4) for $n-1$, and \(\Theta_7\) is a positive constant independent of \(r\).

On the right-hand side of the BAR \eqref{eq:BAR} corresponding to the service completion at station $\ell$, conditional on the pre-jump state $x_-=(z_-,r_{-,e}, r_{-,s})$, the increment of $f_{n,F}$ is
\begin{align*}
	&\E\left[ f_{n,F}(x_+)-f_{n,F}(x_-) \right] = \E\left[ f_{n,F}(x_- + \Delta_{s,\ell}) \right] - f_{n,F}(x_-) \\
	& = \left( \psi_{M-n} + \E\left[\left( {T_{s,\ell}}/{\mu_\ell} \right)^{M-n} \right] \right)\left( \psi_1 + \E\left[ {T_{s,\ell}}/\mu_\ell \right] \right) \Norm{z_{-,\perp} - \delta_\ell}^{n} - \Norm{z_{-,\perp}}^{n} \psi_{M-n}\psi_1 \\
	& \geq \psi_{M-n} \left( \psi_1 + 1/{\alpha^\uu}  \right) \left( \Norm{z_{-,\perp} - \delta_\ell}^{n} - \Norm{z_{-,\perp}}^{n} \right) + \Norm{z_{-,\perp}}^{n} \psi_{M-n} / \mu_\ell \\
	& = \Theta(\norm{z_{-,\perp}}^{n-1}\psi_{M-n+1}) + \Norm{z_{-,\perp}}^{n} \psi_{M-n} /{\mu_\ell},
\end{align*}
where the second equality follows from the fact that $T_{s,\ell}$ and $X^\uu_-$ are independent in Lemma \ref{lmm: independent}, and the last equality follows from \eqref{eq: prop_s2}. Taking the expectation under $\mathbb{P}_{s,\ell}$, the right-hand side of the BAR \eqref{eq:BAR} that corresponds to the service completion at station $\ell$ becomes
\begin{align}
	&\E_{s,\ell}\left[ f_{n,F}(X^\uu_+)-f_{n,F}(X^\uu_-) \right] = \E_{s,\ell}\left[ \E\left[ f_{n,F}(x_+)-f_{n,F}(x_-) \mid X_-^\uu =x_-\right] \right] \notag \\
	&~ = \E_{s,\ell}\left[ \Theta(\norm{Z_{-,\perp}^\uu}^{n-1}\psi_{M-n+1}) \right] +  \E_{s,\ell}\left[ \Norm{Z_{-,\perp}^\uu}^{n} \psi_{M-n} \right]/\mu_\ell \geq -\Theta_8 +   \E_{s,\ell}\left[ \Norm{Z_{-,\perp}^\uu}^{n} \psi_{M-n} \right]/\mu_\ell, \label{eq: RHS2 res4}
\end{align}
where the last inequality holds by the induction hypothesis (S4) for $n-1$, and \(\Theta_8\) is a positive constant independent of \(r\).

In summary, combining \eqref{eq: LHS res4}, \eqref{eq: RHS1 res4} and \eqref{eq: RHS2 res4}, the BAR \eqref{eq:BAR} implies that
\begin{align*}
	[2J + 4J^2M ]E_n \geq \alpha^\uu \Theta_7 +  \E_{e}\left[\Norm{Z_{-,\perp}^\uu}^{n} \psi_{M-n}    \right] + \sum_{\ell\in\J} \lambda^\uu_\ell \left( \Theta_8 +\frac{1}{\mu_\ell} \E_{s,\ell}\left[\Norm{Z_{-,\perp}^\uu}^{n} \psi_{M-n}    \right] \right)\\
	\geq - \alpha^\uu \left( \Theta_7 + \Theta_8 \right) +  \E_{e}\left[\Norm{Z_{-,\perp}^\uu}^{n} \psi_{M-n}    \right] + \sum_{\ell\in\J} \frac{\mu_\ell - r}{\mu_\ell} \E_{s,\ell}\left[\Norm{Z_{-,\perp}^\uu}^{n} \psi_{M-n}    \right],
\end{align*}
where the last inequality follows from $\lambda^\uu_\Sigma = \alpha^\uu$ and \eqref{eq: lambda}.

Therefore, we have
\begin{align*}
	\E_{e}\left[ \Norm{Z_{-,\perp}^\uu}^{n}  \psi_{M-n}  \right] + \sum_{\ell\in\J}  \E_{s,\ell}\left[\Norm{Z_{-,\perp}^\uu}^{n} \psi_{M-n}  \right] & \leq \frac{(2J + 4J^2M )E_n +\alpha^\uu \left( \Theta_7 + \Theta_8 \right) }{\min_{\ell \in \J, r\in (0,r_0)}\frac{\mu_\ell - r}{ \mu_{\ell}}}\\
	& \leq \frac{(2J+4J^2M)E_n + \left( \Theta_3 + \Theta_4 \right) \mu_{\Sigma} }{\frac{\mu_{\min}-r_0}{\mu_{\max}}},
\end{align*}
which completes the proof of (S4) for the induction step $n$. 

\subsection{Proof of Non-integer Case when $\varepsilon>0$} \label{sec: non int}

To prove Proposition \ref{prop: ssc high} when $\varepsilon>0$, it is essential to introduce the following lemma, which aligns directly with Statements (S1)-(S4) in Section \ref{sec: sketch ssc} by replacing $M$ with $\beta\equiv M + \varepsilon / (M+\varepsilon)$ and substituting $n$ with $\beta-1$.

\begin{lemma}\label{prop:general}
Under the moment condition \eqref{eq: moment condition high}, there exist positive and finite constants $C_{\beta-1}$, $D_{\beta-1}$, $E_{\beta-1}$ and $F_{\beta-1}$ independent of $\cc$ such that the following statements hold for all $\cc\in (0,r_0)$: 
\begin{itemize}
	\item[(1)] $\E_{\pi}[\norm{Z_\perp^\uu}^{\beta-1}] \leq C_{\beta-1}$.
	\item[(2)] $\E_{e}[\norm{Z_{-,\perp}^\uu}^{\beta-1}]+\sum_{\ell\in \J}\E_{ {s,\ell}}[\norm{Z_{-,\perp}^\uu}^{\beta-1}] \leq D_{\beta-1}$. 
	\item[(3)] $\E_{\pi}[\norm{Z_\perp^\uu}^{\beta-1}\psi_{1}(R_e^\uu, R_s^\uu)] \leq E_{\beta-1}$.
	\item[(4)] $\E_{e}[\norm{Z_{-,\perp}^\uu}^{\beta-1}\psi_{1}(R_{-,e}^\uu, R_{-,s}^\uu)]+ \sum_{\ell\in \J}\E_{ {s,\ell}}[\norm{Z_{-,\perp}^\uu}^{\beta-1}\psi_{1}(R_{-,e}^\uu, R_{-,s}^\uu)] \leq F_{\beta-1}$.
\end{itemize}
\end{lemma}

To prove Statements (1)-(4), we need to introduce the following lemma, which directly corresponds to the base step of Statements ({S3})-({S4}) when $n=0$.

\begin{lemma}\label{lmm: general}
    Under the moment condition \eqref{eq: moment condition high}, there exist positive and finite constants $A_1$ and $A_2$ independent of $\cc$ such that the following statements hold for all $\cc\in (0,r_0)$: 
    \begin{align}
        &\E_{\pi} \left[\psi_{M+\varepsilon}\left(R_e^{(r)},R_s^{(r)}\right) \right] \leq A_1 \label{eq: S3 general}\\
        & \E_{e}\left[  \psi_{M+\varepsilon}\left(R_{-,e}^{(r)},R_{-,s}^{(r)}\right)\right]+ \sum_{\ell \in \J} \E_{s,\ell}\left[  \psi_{M+\varepsilon}\left(R_{-,e}^{(r)},R_{-,s}^{(r)}\right)\right] \leq A_2\label{eq: S4 general}
    \end{align}
\end{lemma}

\begin{proof}[Proof of Lemma \ref{lmm: general}.]
	The proof of Lemma \ref{lmm: general} follows the same approach as presented in Section~\ref{sec: base}. Specifically, the proof of \eqref{eq: S3 general} can directly adopt the proof approach of Statement ({S3}) in Section~\ref{sec: base} by replacing $M$ with $M+\varepsilon$ in the test function \eqref{eq: func_s3_base}. Similarly, the proof of \eqref{eq: S4 general} can directly adopt the proof approach of Statement ({S4}) in Section~\ref{sec: base} by replacing $M$ with $M+\varepsilon$ in the test function \eqref{eq: func_s4_base}.
\end{proof}

\begin{proof}[Proof of Lemma \ref{prop:general}.]
    Since Proposition \ref{prop: ssc high} for integer cases when $\varepsilon=0$ implies that Statements ({S1}) and ({S2}) hold for $n=M$, Statements ({1}) and ({2}) also hold due to $\beta-1<M$. 
    Statement (3) can be established as follows:
    \begin{equation} \label{eq: proof S3 general}
        \E_{\pi}\left[ \Norm{Z_\perp^\uu}^{\beta-1} \psi_{1}\right] \leq \E_{\pi}\left[ \Norm{Z_\perp^\uu}^{M} \right]^{1-\frac{1}{M+\varepsilon}} \E_{\pi}\left[ \psi_{1}^{M+\varepsilon}\right]^{\frac{1}{M+\varepsilon}} \leq C_{M}^{1-\frac{1}{M+\varepsilon}}\cdot 2J A_1^{\frac{1}{M+\varepsilon}},
    \end{equation}
    where the first inequality holds due to Hölder's inequality, and the second inequality follows from Lemma \ref{lmm: general} and the fact that $\psi_{1}^{M+\varepsilon}  \leq   (2J R_{\max}^\uu )^{M+\varepsilon}  \leq (2J)^{M+\varepsilon}  \psi_{M+\varepsilon}$ with $R_{\max}^\uu\equiv \max(\{R_{e}^{(r)}\}\cup\{R_{s,j}^{(r)}: j\in \J\})$. Consequently, $E_{\beta-1}$ can be set as $2JC_{M}^{1-\frac{1}{M+\varepsilon}} A_1^{\frac{1}{M+\varepsilon}}$. Statement (4) follows the same argument as \eqref{eq: proof S3 general} by replacing $\pi$ to any Palm measure $\nu$ in $\{\mathbb{P}_{e}\}\cup \{\mathbb{P}_{s,\ell};\ell\in \J\}$:
    \begin{equation*}
        \E_{\nu}\left[ \Norm{Z_\perp^\uu}^{\beta-1} \psi_{1}\right] \leq \E_{\nu}\left[ \Norm{Z_\perp^\uu}^{\beta-1} \right]^{1-\frac{1}{M+\varepsilon}} \E_{\nu}\left[ \psi_{1}^{M+\varepsilon}\right]^{\frac{1}{M+\varepsilon}} \leq D_{M}^{1-\frac{1}{M+\varepsilon}}\cdot 2J A_2^{\frac{1}{M+\varepsilon}}.
    \end{equation*}
    Consequently, $F_{k,\beta-1}$ can be set as $4J^2D_{M}^{1-\frac{1}{M+\varepsilon}} A_2^{\frac{1}{M+\varepsilon}}$.  
\end{proof}

\begin{proof}[Proof of Proposition \ref{prop: ssc high}]
    The proof of Proposition \ref{prop: ssc high} parallels the approach in Section \ref{sec: proof of S1}, with only two changes: substituting $n$ with $\beta$ and replacing Statements ({S1})-({S4}) with Statements ({1})-({4}).
\end{proof}

\section{Proof of Steady-state Convergence in Proposition~\ref{prop: weak convergence}} \label{sec:weak}
In this section, we prove the weak convergence of the steady-state average queue length in Proposition \ref{prop: weak convergence}. For simplicity, we equivalently prove the weak convergence of the steady-state total queue length $r\sum_{j\in\J} Z_j^\uu \Rightarrow JZ^*$ as $r\rightarrow 0$. Our proof heavily follows the proof in \citet{BravDaiMiya2023}. 


\subsection{Test Function for the BAR}
In the transform method by \citet{HurtMagu2020}, they utilize the exponential function $g_\theta$ for any $\theta\leq 0$ as
\begin{equation*}
	g_\theta(z) = e^{\theta \sum_{j\in \J} z_j}, \quad z\in \mathbb{Z}_+^J.
\end{equation*}
Incorporating truncation on residual interarrival and service times for $t>0$, we define the test function $f^\uu_{\theta, t}$ for $\theta\leq 0$ and $t>0$ as
\begin{equation*}
	f^\uu_{\theta, t}(x) = g_\theta(z)e^{ - \xi_e(\theta, t) (\alpha^\uu r_e \wedge t^{-1}) - \sum_{j\in \J}  \xi_{s,j}(\theta, t) (\mu_j r_{s,j}\wedge t^{-1})}, ~ x = (z, r_e, r_s) \in \mathbb{S},
\end{equation*}
where $\xi_e(\theta, t)$ and $\xi_{s,j}(\theta, t)$ for $j\in \J$ are the solutions to
\begin{equation} \label{eq: eta_xi}
	e^{\theta} \E_{e} \left[ e^{-\xi_e(\theta, t)( T_e \wedge t^{-1})} \right] = 1, \quad e^{-\theta} \E_{s,j} \left[ e^{-\xi_{s,j}(\theta, t) (T_{s,j}\wedge t^{-1})} \right] = 1 \quad \text{ for } j\in \J.
\end{equation}
The functions $\xi_e(\theta, t)$ and $\xi_{s,j}(\theta, t)$ for $j\in \J$ are uniquely determined by \eqref{eq: eta_xi}, which can be proved in Lemma C.1 of \citet{BravDaiMiya2023}.
Hence, the jump terms in the BAR \eqref{eq:BAR} will be zero, i.e.,
\begin{equation*}
	\E_{e} \left[ f^\uu_{\theta, t}(X^\uu_+)-f^\uu_{\theta, t}(X^\uu_-) \right] = 0, \quad \E_{s,j} \left[ f^\uu_{\theta, t}(X^\uu_+)-f^\uu_{\theta, t}(X^\uu_-) \right] = 0 \quad \text{ for } j\in \J.
\end{equation*}
Therefore, applying the test function $f^\uu_{\theta, t}$ to the BAR \eqref{eq:BAR} yields
\begin{equation} \label{eq: expbar1}
	\E_{\pi}\left[ \A f^\uu_{\theta, t}(X^\uu) \right] = 0.
\end{equation}

We are now ready to define two sets of moment generating functions (MGFs) for $Z^\uu$ and $X^\uu$ as follows. For $Z^\uu$, we define, for each $\theta\leq 0$ and $r\in (0,1)$,
\begin{equation*}
	\phi^\uu(\theta) = \E_{\pi}\left[ g_\theta(Z^\uu) \right] \quad \phi^\uu_j(\theta) = \E_{\pi}\left[ g_\theta(Z^\uu) \mid Z^\uu_j=0 \right] \quad \text{ for } j\in \J,
\end{equation*}
Note that $\phi^\uu(r\eta)$, $\eta<0$, is the MGF of the scaled total queue length $r\sum_{j\in \J}{Z}^\uu_j$.

For $X^\uu$, we set $t= r^{1-\varepsilon_0}$ where $\varepsilon_0\in (0, \delta_0 / (1+\delta_0))$, which is consistent with SSC result in Proposition \ref{prop: ssc}, and define truncated MGFs as
\begin{equation} \label{eq: psi}
	\psi^\uu(\theta) = \E_{\pi}\left[ f^\uu_{\theta, r^{1-\varepsilon_0}}(X^\uu) \right], \quad \psi^\uu_j(\theta) = \E_{\pi}\left[ f^\uu_{\theta, r^{1-\varepsilon_0}}(X^\uu)  \mid Z_j^\uu=0 \right] \quad \text{ for } j\in \J.
\end{equation}

\subsection{Outline of the Proof}

In this section, we will prove the weak convergence of the average queue length in Proposition \ref{prop: weak convergence}. We will follow the steps of the proof in Section 7.1 of \citet{BravDaiMiya2023}.

Using the MGFs $\psi^\uu$ and $\psi^\uu_j$ of \eqref{eq: psi}, we derive an asymptotic BAR for $X^\uu$ as follows, whose proof depends on SSC and is provided in Section \ref{sec: lemmas}.
	\begin{lemma} \label{lmm: asymptotic BAR}
		Assume Assumptions \ref{ass: moment condition} and \ref{ass: stability} hold. Under the JSQ or Po2 policy, for each $\eta\leq0$, as $r\to 0$,
		\begin{align} \label{eq: expbar}
			&\left( \alpha^\uu \xi_e(r\eta, r^{1-\varepsilon_0}) + \sum_{j\in \J}  \mu_j \xi_{s,j}(r\eta, r^{1-\varepsilon_0})  \right) \psi^\uu(r\eta) \notag \\
			&\qquad - \sum_{j\in \J} \mu_j \mathbb{P}\left(Z_j^\uu=0\right)  \xi_{s,j}(r\eta, r^{1-\varepsilon_0}) \psi^\uu_j(r\eta) = o(r^2),
		\end{align}
  where $f(r)=o(r)$ denotes $\lim_{r\to 0} f(r)/r = 0$.
	\end{lemma}

To rewrite \eqref{eq: expbar} in a more tractable form, we need asymptotic expansions of $\xi_e$ and $\xi_{s,j}$ for $j\in \J$. Specifically, for each fixed $\eta\leq 0$ and $j\in \J$, we have, as $r\to 0$,
	\begin{align}
		\xi_e(r\eta, r^{1-\varepsilon_0}) &= r\eta+\frac{1}{2}c_e^2 r^2\eta^2 + o(r^2),\quad \xi_{s,j}(r\eta, r^{1-\varepsilon_0}) = -r\eta+\frac{1}{2}c_{s,j}^2 r^2\eta^2 + o(r^2), \label{eq: eta_xi_expansion}
	\end{align}
	where $c_e^2$ and $c_{s,j}^2$ are the SCV of interarrival and service times, respectively. The proof of \eqref{eq: eta_xi_expansion} is given in Lemma 7.5 of \citet{BravDaiMiya2023}.
Therefore, due to $0\leq \mu_j-\lambda_j^\uu\leq r$ in \eqref{eq: lambda}, the asymptotic BAR \eqref{eq: expbar}, for each $\eta\leq 0$, as $r\to 0$, becomes
	\begin{align*}
		&\left( r\alpha^\uu\eta+\frac{1}{2}\alpha^\uu c_e^2 r^2\eta^2 + \sum_{j\in \J}   \left( -r\mu_j\eta+\frac{1}{2}\mu_jc_{s,j}^2 r^2\eta^2 \right)  \right) \psi^\uu(r\eta) \\
		&\qquad - \sum_{j\in \J}  \left( \mu_j- \lambda^\uu_j \right) \left( -r\eta \right) \psi^\uu_j(r\eta) = o(r^2).
	\end{align*}
	According to the heavy traffic condition in \eqref{eq: heavy traffic} and dividing by $r^2\eta$, we have, for each $\eta\leq 0$ and as $r\to 0$,
	\begin{align*}
		&\left( -1+\eta \left( \frac{1}{2}\alpha^\uu c_e^2  + \sum_{j\in \J}\frac{1}{2}\mu_jc_{s,j}^2\right)  \right)   \psi^\uu(r\eta) + \sum_{j\in \J} \frac{ \mu_j- \lambda^\uu_j }{r} \psi^\uu_j(r\eta) = o(1).
	\end{align*}

Using the SSC in Proposition \ref{prop: ssc} and the expansions in \eqref{eq: eta_xi_expansion}, we have, as $r\to 0$,
	\begin{equation*}
		\psi^\uu(r \eta) - \phi^\uu(r\eta) = o(1), \quad \psi^\uu_j(r \eta) - \phi_j^\uu (r \eta) = o(1) \quad \text{ for } j\in \J,
	\end{equation*}
	whose proof is given in (7.28) of \citet{BravDaiMiya2023}.
	Hence, we can replace $\psi^\uu$ and $\psi^\uu_j$ by $\phi^\uu$ and $\phi_j^\uu$ in the asymptotic BAR \eqref{eq: expbar}, which implies that
	\begin{equation*}
		\left( -1+\eta \left( \frac{1}{2}\alpha^\uu c_e^2  + \sum_{j\in \J}\frac{1}{2}\mu_jc_{s,j}^2\right)     \right)   \phi^\uu(r\eta) + \sum_{j\in \J} \frac{ \mu_j- \lambda^\uu_j }{r} \phi_j^\uu(r\eta) = o(1).
	\end{equation*}

According to SSC in Proposition \ref{prop: ssc}, we have the following lemma, and the proof is given in \ref{sec: lemmas}.
	\begin{lemma} \label{lmm: one}
		Assume Assumptions \ref{ass: moment condition} and \ref{ass: stability} hold. Under the JSQ or Po2 policy, we have 
		\begin{equation*}
			\lim_{r\to 0} \sum_{j\in \J} \frac{ \mu_j- \lambda^\uu_j }{r} \phi_j^\uu(r\eta) = 1.
		\end{equation*}
	\end{lemma}
	Hence, the limit of $\phi^\uu(r\eta)$ is
	\begin{equation*}
		   \lim_{r\to 0} \phi^\uu(r\eta) =   \frac{1}{1-\eta \left( \frac{1}{2}\alpha^\uu c_e^2  + \sum_{j\in \J}\frac{1}{2}\mu_jc_{s,j}^2\right) }   = \frac{1 }{ 1-\eta Jm } ,
	\end{equation*}
	which is the MGF of an exponential random variable with mean $Jm$, and $m$ is defined in \eqref{eq: mean}.

\subsection{Proof of Technical Lemmas} \label{sec: lemmas}

\begin{proof}[Proof of Lemma \ref{lmm: asymptotic BAR}]
Setting $\theta$ to $r\eta$, the BAR \eqref{eq: expbar1} implies that
\begin{align*}
	&\left( \alpha^\uu \xi_e(r\eta, r^{1-\varepsilon_0}) +  \sum_{j\in \J} \mu_j \xi_{s,j}(r\eta, r^{1-\varepsilon_0})  \right)\psi^\uu(r\eta)  - \sum_{j\in \J} \mu_j \mathbb{P}\left(Z_j^\uu=0\right)\xi_{s,j}(r\eta, r^{1-\varepsilon_0})   \psi^\uu_j(r\eta) \\
	& \quad +  \alpha^\uu  \xi_e(r\eta, r^{1-\varepsilon_0})  \E_{\pi}\left[ f_{r\eta, r^{1-\varepsilon_0}}(X^\uu)     \I{\alpha^\uu R_e^\uu > r^{\varepsilon_0-1}}     \right] \\
	& \quad +  \sum_{j\in \J} \mu_j \xi_{s,j}(r\eta, r^{1-\varepsilon_0})    \E_{\pi}\left[ f_{r\eta, r^{1-\varepsilon_0}}(X^\uu)   \I{Z^\uu_j>0} \I{\mu_j R^\uu_{s,j} > r^{\varepsilon_0-1}}  \right] = 0.
\end{align*}
Since the last two terms have the order of $o(r^2)$ by Lemma 8.2 of \citet{BravDaiMiya2023}, the proof is completed.
\end{proof}

\begin{proof}[Proof of Lemma \ref{lmm: one}]
	From the SSC in Proposition \ref{prop: ssc} and $\mathbb{P}( Z^\uu_j=0)=(\mu_j- \lambda^\uu_j)/\mu_j$ in \eqref{eq: lambda}, for each $j\in\J$, we have
	\begin{align*}
		\abs{1 - \phi_j^\uu(r\eta) } &= \E_{\pi}\left[1 - e^{r\eta \sum_{k\in \J}Z^\uu_k}  \mid Z^\uu_j=0 \right] \leq r\abs{\eta}\E_{\pi}\left[ \sum_{k\in \J}Z^\uu_k \mid Z^\uu_j=0 \right]\\
		& = J r\abs{\eta}\E_{\pi}\left[ \bar{Z}^\uu - Z^\uu_j \mid Z^\uu_j=0 \right] \leq J r\abs{\eta}\E_{\pi}\left[ \norm{Z_\perp^\uu} \mid Z^\uu_j=0 \right],
	\end{align*}
        where the first inequality follows from $1-e^{-x}\leq x$ for all $x\geq 0$.
        By $\sum_{j\in \J} (\mu_j - \lambda_j^\uu)=r$ and $\mathbb{P}( Z^\uu_j=0)=(\mu_j- \lambda^\uu_j)/\mu_j$ in \eqref{eq: lambda}, we have
        \begin{align*}
     &\abs{1 - \sum_{j\in \J} \frac{ \mu_j- \lambda^\uu_j }{r} \phi_j^\uu(r\eta)} = \sum_{j\in \J} \frac{ \mu_j- \lambda^\uu_j }{r}\abs{1 -  \phi_j^\uu(r\eta)}= \sum_{j\in \J} \frac{ \mu_j\Prob\left(Z^\uu_j = 0\right) }{r}\abs{1 -  \phi_j^\uu(r\eta)} \\
     & \leq \sum_{j\in \J} \frac{ \mu_j\Prob\left(Z^\uu_j = 0\right) }{r}J r\abs{\eta}\E_{\pi}\left[ \norm{Z_\perp^\uu} \mid Z^\uu_j=0 \right] = \sum_{j\in \J}\mu_j J \abs{\eta}\E_{\pi}\left[ \norm{Z_\perp^\uu} \I{Z^\uu_j=0} \right] \\
		& \leq   \sum_{j\in \J}\mu_j J \abs{\eta}\E_{\pi}\left[ \norm{Z_\perp^\uu}^{1+\varepsilon_0} \right]^{\frac{1}{1+\varepsilon_0}}\mathbb{P}\left( Z^\uu_j=0 \right)^{\frac{\varepsilon_0}{1+\varepsilon_0}}\leq   \sum_{j\in \J}\mu_j^{\frac{1}{1+\varepsilon_0}} J \abs{\eta}\E_{\pi}\left[ \norm{Z_\perp^\uu}^{1+\varepsilon_0} \right]^{\frac{1}{1+\varepsilon_0}}r^{\frac{\varepsilon_0}{1+\varepsilon_0}}\\
		&=o(1),
 \end{align*}
	where the second inequality holds due to Hölder's inequality and the final inequality follows from $\mathbb{P}( Z^\uu_j=0)=(\mu_j- \lambda^\uu_j)/\mu_j\leq r/\mu_j$ in \eqref{eq: lambda}.
\end{proof}

\bibliographystyle{ACM-Reference-Format}
\bibliography{reference}

\appendix

\section{Proof of Lemma \ref{lmm: prop_eta}} \label{sec: lemma}
In this section, we provide the proofs of Lemma \ref{lmm: prop_eta}, which is key to the proof of the induction hypotheses (S1)-(S4). We will first prove the JSQ case and then extend the proof to the Po2 case.

Our proof highly depends on the mean value theorem. For any $z\in \R^J$, the gradient of $\norm{z}^n$ is given by
\begin{equation*}
	\nabla \norm{z} = \frac{z}{\norm{z}}, \quad \nabla \norm{z}^n = n \norm{z}^{n-2}z.
\end{equation*}
Hence, the mean value theorem implies that for any $z\in \Z_+^J$ and $\delta\in \R^J$ with $\norm{\delta}<1$, there exists $\xi\in (0,1)$ such that
\begin{equation} \label{eq: diff}
	\norm{z_\perp + \delta}^n - \norm{z_\perp}^n = n \norm{z_\perp + \xi\delta}^{n-2} \left( z_\perp + \xi\delta \right)^\T \delta = \Theta(\norm{z_\perp}^{n-2})\Theta(\norm{z_\perp}) = \Theta(\norm{z_\perp}^{n-1}),
\end{equation}
where $z_\perp^\T \delta=\Theta(\norm{z_\perp})$ holds due to $\abs{z_\perp^\T \delta} \leq \norm{z_\perp} \norm{\delta}$.

\paragraph*{JSQ Policy}
Since the JSQ policy selects the queue with the shortest job, we have $z_j p_{j\mid z} = z_{\min}p_{j\mid z}$. Hence, \eqref{eq: prop_e} can be proved as follows. The mean value theorem implies that there exists $\xi_{e,j}\in (0,1)$ such that
\begin{align*}
	&\sum_{j\in \J}\frac{1}{n+1}\norm{z_{\perp}+\delta_j}^{n+1}p_{j\mid z}  - \frac{1}{n+1}\norm{z_{\perp}}^{n+1} - \norm{z_{\perp}}^{n-1} z_{\perp, \min}\\
	&\quad = \sum_{j\in \J}\left( \frac{1}{n+1}\norm{z_{\perp}+\delta_j}^{n+1}- \frac{1}{n+1}\norm{z_{\perp}}^{n+1} \right) p_{j\mid z}  - \norm{z_{\perp}}^{n-1} z_{\perp, \min}\\
	&\quad = \sum_{j\in \J} \norm{z_{\perp}+\xi_{e,j}\delta_j}^{n-1} \left( z_{\perp}+\xi_{e,j}\delta_j \right)^\T \delta_j p_{j\mid z}  - \norm{z_{\perp}}^{n-1} z_{\perp, \min}\\
	&\quad = \sum_{j\in \J} \norm{z_{\perp}+\xi_{e,j}\delta_j}^{n-1} \left( z_{\perp,\min}+\xi_{e,j}\norm{\delta_j}^2 \right) p_{j\mid z}  - \norm{z_{\perp}}^{n-1} z_{\perp, \min}\\
	&\quad = \sum_{j\in \J} \left( \norm{z_{\perp}+\xi_{e,j}\delta_j}^{n-1} - \norm{z_{\perp}}^{n-1}  \right)  z_{\perp,\min} p_{j\mid z} + \sum_{j\in \J} \norm{z_{\perp}+\xi_{e,j}\delta_j}^{n-1} \xi_{e,j}\norm{\delta_j}^2  p_{j\mid z}\\
	&\quad = O(\norm{z_{\perp}}^{n-2}) O(\norm{z_{\perp}})  + O(\norm{z_{\perp}}^{n-1}) = O(\norm{z_{\perp}}^{n-1}),
\end{align*}
where the third equality follows from the fact that $z_{\perp}^\T\delta_j = z_{\perp}^\T\e{j} = z_{\perp,j}$ and $z_{\perp,j}p_{j\mid z} = z_{\perp,\min}p_{j\mid z}$, and the fifth equality holds due to \eqref{eq: diff} and $z_{\perp,\min}\leq \norm{z_\perp}$. The proof of \eqref{eq: prop_s} is similar but more simpler. The mean value theorem implies that there exists $\xi_{s,j}\in (0,1)$ such that
\begin{align*}
	&\frac{1}{n+1}\norm{z_{\perp}-\delta_j}^{n+1} - \frac{1}{n+1}\norm{z_{\perp}}^{n+1} + \norm{z_{\perp}}^{n-1} z_{\perp, j} \\
	&\quad = \norm{z_{\perp}-\xi_{s,j}\delta_j}^{n-1}\left( z_{\perp}-\xi_{s,j}\delta_j \right)^\T \left( -\delta_j \right)+ \norm{z_{\perp}}^{n-1} z_{\perp, j}\\
	&\quad = \norm{z_{\perp}-\xi_{s,j}\delta_j}^{n-1}\xi_{s,j}\norm{\delta_j}^2-\norm{z_{\perp}-\xi_{s,j}\delta_j}^{n-1}z_{\perp,j}+ \norm{z_{\perp}}^{n-1} z_{\perp, j}\\
	&\quad =O(\norm{z_{\perp}}^{n-1}) + O(\norm{z_{\perp}}^{n-2}) O(\norm{z_{\perp}}) = O(\norm{z_{\perp}}^{n-1}),
\end{align*}
where the second equality follows from the fact that $z_{\perp}^\T\delta_j = z_{\perp}^\T\e{j} = z_{\perp,j}$, and the third equality holds due to \eqref{eq: diff} and $z_{\perp,j}\leq \norm{z_\perp}$.

The conclusions of \eqref{eq: prop_e2} and \eqref{eq: prop_s2} are the direct consequence of \eqref{eq: diff}, so we omit the proof here. Now, we are ready to prove \eqref{eq: coeffS1}. 
Since $\alpha^\uu = \mu_{\Sigma} - r$ in \eqref{eq: heavy traffic}, we have
\begin{align*}
	-\A f_n(x) &= -\eta_{n,e}(z_{\perp})  \alpha^\uu  + \sum_{j\in \J} \eta_{n,s,j}(z_{\perp}) \geq  -\left( \mu_\Sigma - r \right) \norm{z_{\perp}}^{n-1} \left( z_{\min} - \bar{z} \right)   + \sum_{j\in \J} \norm{z_{\perp}}^{n-1} \left(z_j- \bar{z}\right)\mu_j   \\
	&= \norm{z_{\perp}}^{n-1} \left( -r  \left(   \bar{z} - z_{\min} \right) -   \sum_{j\in \J} \mu_j  \left( z_{\min} - \bar{z} \right)   + \sum_{j\in \J} \left(z_j- \bar{z}\right)\mu_j \right)   \\
	&=  \norm{z_{\perp}}^{n-1} \left( -r  \left(   \bar{z} - z_{\min} \right)  + \sum_{j\in \J} \left(z_j-z_{\min} \right)\mu_j  \right)    \\
	&\geq \norm{z_{\perp}}^{n-1} \left( -r  \left(   z_{\max} - z_{\min} \right)  + \left( z_{\max} - z_{\min}\right) \mu_{\operatorname{argmax}_{k\in \J} z_{k}}   \right)  \\
	& = \norm{z_{\perp}}^{n-1} \left( \mu_{\min} - r \right) \left({z}_{\max}- z_{\min}\right)  \geq \frac{\mu_{\min}}{2\sqrt{J}} \Norm{z_{\perp}}^n ,
\end{align*}
where the first equality follows from the fact that
\begin{equation} \label{eq: z indicator}
	z_{\perp,j}   \I{z_j>0} = \left( z_j - \bar{z} \right)   \I{z_j>0} \geq    \left( z_j - \bar{z} \right)    =    z_{\perp,j},
\end{equation} the second inequality holds by $z_{\max} \equiv \max_{j\in \J}z_j \geq \bar{z}$ and neglecting all terms except the term $z_{\max}$, and the last inequality holds due to $r\in (0, r_0)$ with $r_0\equiv \mu_{\min}/2$ for JSQ policy and
\begin{equation} \label{eq: z_perp max min}
	\Norm{z_{\perp}} = \sqrt{\sum_{j\in \J} \left( z_{j} - \bar{z} \right)^2} \leq \sqrt{\sum_{j\in \J} \left( z_{\max} - z_{\min} \right)^2} = \sqrt{J} \left( z_{\max} - z_{\min} \right).
\end{equation}
Hence, \eqref{eq: coeffS1} holds for JSQ policy with $b = \mu_{\min}/2\sqrt{J}$.

\paragraph{Po2 Policy} Under the Po2 policy, the decision first selects two stations, \( j_1 \) and \( j_2 \) with probability \( p_{j_1j_2} \), and then routes the external arrival to the station with the minimum queue length \(\operatorname{argmin}_{j\in \{j_1,j_2\}} z_j\). Therefore, we have
\begin{align*}
	\sum_{j\in \J} z_{\perp,j}  p_{j|z} =\sum_{j\in \J} z_{j}  p_{j|z} - \bar{z} = \sum_{j_1<j_2} \min\{z_{j_1},z_{j_2}\} p_{j_1j_2} - \bar{z}.
\end{align*}
Applying $\min\{z_{\min}, z_{\max}\}= (z_{\min} + z_{\max})/2 - (z_{\max} - z_{\min})/2$ for the minimum and maximum queue lengths and $\min\{z_{j_1}, z_{j_2}\}\leq (z_{j_1} + z_{j_2})/2$ for all other $j_1,j_2\in \J$, we have
\begin{align}
	\sum_{j\in \J} z_{\perp,j}  p_{j|z} &\leq \sum_{j_1<j_2}  \frac{p_{j_1j_2}}{2}(z_{j_1} + z_{j_2}) - \frac{\min_{j_1,j_2\in \J} p_{j_1j_2}}{2}(z_{\max} - z_{\min})- \bar{z}\notag \\
	& = \sum_{j\in \J}  {p_{j}}z_{j}  - \frac{\min_{j_1,j_2\in \J} p_{j_1j_2}}{2}(z_{\max} - z_{\min}) - \bar{z} \leq \sum_{j\in \J}  {p_{j}}z_{\perp,j}  - \frac{p_{\min}^2}{\sqrt{J}}\Norm{z_{\perp}}, \label{eq: ub po2}
\end{align}
where the last inequality follows from $\min_{j_1,j_2\in \J} p_{j_1j_2}\leq 2p_{\min}^2$ and \eqref{eq: z_perp max min}. Note that the choice of $\eta_{n,e}(z_{\perp})$ is decided by \eqref{eq: ub po2}.

The mean value theorem implies that there exists $\xi_{e,j}\in (0,1)$ such that
\begin{align*}
	&\sum_{j\in \J}\frac{1}{n+1}\norm{z_{\perp}+\delta_j}^{n+1}p_{j\mid z}  - \frac{1}{n+1}\norm{z_{\perp}}^{n+1} - \eta_{n,e}(z_{\perp})\\
	&\quad = \sum_{j\in \J}\left( \frac{1}{n+1}\norm{z_{\perp}+\delta_j}^{n+1}- \frac{1}{n+1}\norm{z_{\perp}}^{n+1} \right) p_{j\mid z}  - \eta_{n,e}(z_{\perp})\\
	&\quad = \sum_{j\in \J} \norm{z_{\perp}+\xi_{e,j}\delta_j}^{n-1} \left( z_{\perp}+\xi_{e,j}\delta_j \right)^\T \delta_j p_{j\mid z}  - \eta_{n,e}(z_{\perp})\\
	&\quad = \sum_{j\in \J} \norm{z_{\perp}+\xi_{e,j}\delta_j}^{n-1} \xi_{e,j}\norm{\delta_j}^2  p_{j\mid z} + \sum_{j\in \J} \norm{z_{\perp}+\xi_{e,j}\delta_j}^{n-1} z_{\perp, j} p_{j\mid z}   - \eta_{n,e}(z_{\perp})\\
	&\quad = \sum_{j\in \J} \norm{z_{\perp}+\xi_{e,j}\delta_j}^{n-1} \xi_{e,j}\norm{\delta_j}^2  p_{j\mid z} + \sum_{j\in \J} \left( \norm{z_{\perp}+\xi_{e,j}\delta_j}^{n-1}  - \norm{z_{\perp}}^{n-1}\right) z_{\perp, j} p_{j\mid z} \\
	&\qquad + \sum_{j\in \J} \norm{z_{\perp}}^{n-1} z_{\perp, j} p_{j\mid z}   - \eta_{n,e}(z_{\perp})\\
	&\quad \leq \sum_{j\in \J} \norm{z_{\perp}+\xi_{e,j}\delta_j}^{n-1} \xi_{e,j}\norm{\delta_j}^2  p_{j\mid z} + \sum_{j\in \J} \left( \norm{z_{\perp}+\xi_{e,j}\delta_j}^{n-1}  - \norm{z_{\perp}}^{n-1}\right) z_{\perp, j} p_{j\mid z} \\
	&\quad =  O(\norm{z_{\perp}}^{n-1}) + O(\norm{z_{\perp}}^{n-2}) O(\norm{z_{\perp}})   = O(\norm{z_{\perp}}^{n-1}),
\end{align*}
where the third equality follows from the fact that $z_{\perp}^\T\delta_j = z_{\perp}^\T\e{j} = z_{\perp,j}$, the inequality holds by \eqref{eq: ub po2}, and the fifth equality holds due to \eqref{eq: diff} and $z_{\perp,\min}\leq \norm{z_\perp}$. 

The proofs of \eqref{eq: prop_s}, \eqref{eq: prop_e2} and \eqref{eq: prop_s2} are the same as the JSQ case, so we omit them here. Now, we are ready to prove \eqref{eq: coeffS1} for the Po2 policy. Since $\alpha^\uu = \mu_{\Sigma} - r$ in \eqref{eq: heavy traffic}, we have
\begin{align*}
	-\A f_n(x) &= -\eta_{n,e}(z_{\perp}) \cdot \alpha^\uu  + \sum_{j\in \J} \eta_{n,s,j}(z_{\perp})\\
	&\geq  -\left( \mu_\Sigma - r \right) \norm{z_{\perp}}^{n-1} \left( \sum_{j\in \J}  {p_{j}}z_{\perp,j}  - \frac{p_{\min}^2}{\sqrt{J}}\Norm{z_{\perp}} \right)   + \norm{z_{\perp}}^{n-1}\sum_{j\in \J}  z_{\perp,j}\mu_j   \\
	&= \left( \mu_\Sigma - r \right) \norm{z_{\perp}}^{n-1} \cdot \frac{p_{\min}^2}{\sqrt{J}}\Norm{z_{\perp}} -\left( \mu_\Sigma - r \right) \norm{z_{\perp}}^{n-1}  \sum_{j\in \J}  {p_{j}}z_{\perp,j}     + \norm{z_{\perp}}^{n-1}\sum_{j\in \J}  z_{\perp,j}\mu_j   \\
	&= \left( \mu_\Sigma - r \right) \norm{z_{\perp}}^{n-1} \cdot \frac{p_{\min}^2}{\sqrt{J}}\Norm{z_{\perp}} +r \norm{z_{\perp}}^{n-1}  \sum_{j\in \J}  {p_{j}}z_{\perp,j}    \\
	&\geq \left( \mu_\Sigma - r \right) \frac{p_{\min}^2}{\sqrt{J}}\Norm{z_{\perp}}^n - r \norm{z_{\perp}}^{n}   = \left(\mu_\Sigma - \left( 1 + \frac{\sqrt{J}}{p_{\min}^2} \right) r \right)  \frac{p_{\min}^2}{\sqrt{J}}\Norm{z_{\perp}} \\
	&\geq \frac{\mu_{\Sigma}}{2}  \frac{p_{\min}^2}{\sqrt{J}}\Norm{z_{\perp}}^n,
\end{align*}
where the first inequality holds due to \eqref{eq: z indicator}, the third equality follows from the fact that $p_{j} = \mu_j/\mu_\Sigma$, the second inequality holds by $z_{\perp,j} \geq -\norm{z_{\perp}}$, and the last inequality holds due to $r\in (0, r_0)$ where $r_0 \leq \mu_\Sigma/(1 + \frac{p_{\min}^2}{\sqrt{J}})$ for Po2 policy. Hence, \eqref{eq: coeffS1} holds for Po2 policy with $b = \mu_\Sigma p_{\min}^2/(2\sqrt{J})$ and the proof of Lemma \ref{lmm: prop_eta} is completed.

\end{document}